\documentclass[10pt,oneside,english,reqno]{amsart}

\usepackage{ifmtarg}
\usepackage{amsmath,amssymb,amsthm}
\usepackage[utf8]{inputenc}
\usepackage[T1]{fontenc}
\usepackage{hyperref}
\usepackage{cite}
\usepackage{color}

\hypersetup{
breaklinks=true,
colorlinks=true,
linkcolor=blue,
citecolor=blue,
urlcolor=blue,
}

\newtheorem{thm}{Theorem}[section]

\newtheorem{lem}[thm]{Lemma}
\newtheorem{cor}[thm]{Corollary}
\theoremstyle{definition}
\newtheorem{defn}[thm]{Definition}

\theoremstyle{remark}

\numberwithin{equation}{section}

\newcommand{\R}{\mathbb{R}}  
\newcommand{\C}{\mathbb{C}} 

\newcommand{\Z}{\mathbb{Z}} 
\newcommand{\N}{\mathbb{N}} 
\newcommand{\E}{\mathcal{E}}
\newcommand{\Ell}{\mathcal{L}}
\newcommand{\D}{\mathcal{D}}

\newcommand{\RZ}{\R / \Z}
\newcommand{\Si}{\textnormal{Si}}

\usepackage[
initials]{amsrefs} 

\AtBeginDocument{%
   \def\MR#1{}
}

\makeatletter
\renewcommand\section{\@startsection{section}{1}%
  \z@{.7\linespacing\@plus\linespacing}{.5\linespacing}%
  {\normalfont\bfseries\centering}}
\makeatother

\newenvironment{nouppercase}{%
  \renewcommand{\uppercasenonmath}[1]{}}{}

\begin{document}
	
\date{\today}

\title[On the regularity of critical points for O'Hara's knot energies: From smoothness to analyticity.]{\Large On the regularity of critical points for O'Hara's knot energies: From smoothness to analyticity.}

\author{Nicole Vorderobermeier}
\thanks{The author acknowledges support by the Austrian Science Fund (FWF), Grant P 29487.}
\address{Fachbereich Mathematik, Universität Salzburg, Hellbrunner Strasse 34, 5020 Salzburg, Austria}
\email{nicole.vorderobermeier@sbg.ac.at}
\urladdr{https://uni-salzburg.at/index.php?id=209722} 


\begin{abstract}
We prove the analyticity of smooth critical points for O'Hara's knot energies 
$\E^{\alpha,p}$, with $p=1$ and $2<\alpha< 3$, subject to a fixed length constraint. 
This implies, together with the main result in \cite{BR13}, that bounded energy critical points of $\E^{\alpha,1}$ 
subject to a fixed length constraint are not only $C^\infty$ but also analytic. 
Our approach is based on Cauchy's method of majorants and a decomposition of the gradient 
that was adapted from the Möbius energy case $\E^{2,1}$ in \cite{BV19}.
\end{abstract}

\begin{nouppercase}
\maketitle
\end{nouppercase}

\section{Introduction}

Knots have always played an important role in arts and crafts, commerce and trade, as well as in our everyday life. Therefore they naturally became a topic of interest for mathematicians. 
During the 19th century  the study of knots strongly influenced the development of topology \cite{TvdG96}. 
Today the theory of knots appears in several branches of mathematics such as in calculus of variations, geometric analysis, topology as well as in applications to modern quantum physics (e.g. \cite{K05}) and biochemistry (e.g. protein molecules \cite{KS98} or DNA \cite{CKS98,GM99}).

A knot in mathematical terms is a Jordan curve in the three-dimensional Euclidean space (i.e.~a continuous embedding of the circle $\RZ$ into $\R^3$). We say that two given knots belong to the same knot class if one knot can be deformed into the other without any  `cuttings and gluings' nor any self-intersections. 
Within this context, the following two questions arise: Is it possible to determine `nicely' shaped representatives of each knot class? And if so, how `nice' are they?

The first question was originally addressed by  Fukuhara \cite{F88} in the context of  polygonal knots.
In order to detect optimal shapes of a polygonal knot, an energy modelling a form of self-avoidance on the space of polygonal knots was introduced---from which optimal shapes can be identified as energy minimizers.
Subsequently, O’Hara \cite{OH91,OH92-2} extended Fukuhara's approach to geometric knots. For any Jordan curve $\gamma: \RZ\rightarrow\R^3$, he introduced the potential energies 
\begin{align}\label{defn:O'Hara's knot energies}
\E^{\alpha,p}(\gamma) = \iint_{(\RZ)^2} \left(\frac{1}{|\gamma(x)-\gamma(y) |^\alpha} - \frac{1}{\mathcal{D}(\gamma(x),\gamma(y))^\alpha} \right)^p |\dot{\gamma}(x)| |\dot{\gamma}(y)| \mathrm{d}x \mathrm{d}y.
\end{align}
Here the quantity $\mathcal{D}(\gamma(x),\gamma(y))$ denotes the intrinsic distance between the points $\gamma(x)$ and $\gamma(y)$ along the curve, i.e. for $|x-y|\leq \frac{1}{2}$ we have 
\begin{align*}
\mathcal{D}(\gamma(x),\gamma(y)) = \min \{\Ell (\gamma|_{[x,y]}), \Ell(\gamma) - \Ell (\gamma|_{[x,y]}) \},
\end{align*}
where $\Ell (\gamma) = \int_0^1 |\dot{\gamma} (t)| \mathrm{d}t$ denotes the length of the curve $\gamma$. 
All values of $\E^{\alpha,p}$ are non-negative due to the fact that the intrinsic distance between two points of the curve is always greater than the Euclidean distance. The factor $|\gamma'(x)| |\gamma'(y)|$ guarantees the invariance of $\E^{\alpha,p}$ under reparametrization of the curve. In addition, $\E^2=\E^{2,1}$ is called Möbius energy since it is invariant under Möbius transformations (cf. \cite[Theor.~2.1]{FHW94}).
For $p \geq 1$  and $0 < \alpha p < 2p +1$ the energies $\E^{\alpha,p}$ are globally minimized by circles, whereas for $\alpha p \geq 2p +1$ their values become infinite for every closed regular curve (cf.~\cite[Corol.~3]{ACFGH03}). 

To distinguish between knot classes, it is desirable  for a  knot energy to be self-repulsive (i.e.~ to penalize self-intersections) and tight (i.e.~ to blow up on a sequence of small knots that pull tight). O'Hara   showed that the knot energies $\E^{\alpha,p}$ are indeed self-repulsive for the cases	$p \geq \tfrac 2\alpha$, $0< \alpha \leq 2$ and $\tfrac{1}{\alpha-2}> p \geq  \tfrac 2\alpha$, $2< \alpha \leq 4$ (cf.~\cite[Theor.~1.1]{OH94}) and  tight if and only if $\alpha p >2$ (cf.~\cite[Theor.~3.1]{OH92-2}). Moreover,   \cite[Theor.~3.2]{OH94} showed that for $\alpha p > 2$ there exist minimizers of the energies within every knot class among all curves with fixed length. In case of the Möbius energy, we also have self-repulsiveness but not tightness (cf.~\cite[Theor.~3.1]{OH94}) and we only know that there exist minimizers in prime knot classes (cf.~\cite[Theor.~4.3]{FHW94}). We see that due to the well-definedness and existence of minimizers of a certain range of O'Hara's energies $\E^{\alpha,p}$, this approach for answering our first question looks promising.

In the following we want to focus on the second question by determining the regularity properties of minimizers of the knot energies.
For that reason we restrict ourselves to examining O'Hara's energies $\E^\alpha = \E^{\alpha,1}$ for $2\leq \alpha <3$ as they are well-defined in a knot-theoretic sense and have a  non-degenerate first variation (in contrast to the $p>1$ case).
For the Möbius energy, one of the first regularity results goes back to He \cite[Chap.~5]{H00} who states, based on Freedman, He and Wang's work \cite[Theor.~5.4]{FHW94}, that any local minimizer $\gamma$ of $\E^2$ with respect to the $L^\infty$-topology belongs to  $C^\infty (\RZ,\R^3)$.
Reiter \cite{R12} generalized this result to the class of O'Hara's knot energies $\E^\alpha$ for  $2\leq \alpha <3$ and $n\geq 3$ by showing that any critical point $\gamma$ of $\E^\alpha$ in the class $H^\alpha(\RZ,\R^n)$ with $\gamma ''\in L^2(\RZ,\R^n)$ is smooth.  
A subsequent improvement on the energy space conditions furnishes the following  $C^\infty$-result.

\begin{thm}[Blatt and Reiter \cite{BR13}] \label{smooth}
	Let $\gamma:\RZ\rightarrow \R^n$ be a simple closed Lipschitz-continuous  arc-length parametrized  curve with 
	$\gamma \in H^{\frac{\alpha +1}{2},2}(\mathbb R / \mathbb Z, \mathbb R^n)$ for any $2< \alpha<3$. If $\gamma$ is a critical point of $\E^\alpha + \lambda \mathcal{L}$, i.e. 
	\begin{align*}
	\delta \E^\alpha(\gamma;h) + \lambda \int_{\RZ} \langle \dot{\gamma}, \dot{h} \rangle \mathrm{d} x = 0 
	\quad \text{for all} \quad h\in H^{\frac{\alpha +1}{2},2}(\RZ,\R^n),
	\end{align*}
	then $\gamma$ belongs to $C^\infty (\RZ,\R^n)$.
\end{thm}

An important characterization of curves with finite O'Hara energy goes back to Blatt \cite{B12} who showed that curves $\gamma$ have finite energy $\E^{\alpha}(\gamma) < \infty$ if and only if they belong to the fractional Sobolev space $H^{\frac{\alpha +1}{2},2}(\mathbb R / \mathbb Z, \mathbb R^n)$. This result, combined with Theorem~\ref{smooth}, raises the question whether critical points of O'Hara's energies $\E^\alpha$ of finite energy are not only smooth but also analytic. 
For the Möbius energy $\E^2$ this question was solved in the affirmative by \cite{BV19}.
One might conjecture that it is possible to transfer this result to the energy classes  $\E^\alpha$ for  $2< \alpha<3$.
Unlike the Möbius energy case, the first variation of $\E^\alpha$ for  $2< \alpha<3$ leads to the appearance of a fractional derivative of the product of two functions.
The latter is subsequently analyzed via a new fractional Leibniz rule instead of the bilinear Hilbert transform, which was used in the Möbius energy case. 

The main result of this article confirms the above analyticity conjecture, namely: 

\begin{thm}\label{main}
Let $\gamma :\RZ\rightarrow \R^n$ be a closed simple arc-length parametrized curve in
$C^\infty(\RZ,\R^n)$. 
If $\gamma$ is a critical point of O'Hara's energy  $\E^\alpha$,  $2< \alpha<3$,
with a length term $\Ell$,  
i.e.~$\E^\alpha + \lambda \Ell$, then the curve $\gamma$ is analytic.
\end{thm}

This theorem, together with the characterization of the energy spaces in \cite{B12} 
and  Theorem~\ref{smooth},  implies the following:

\begin{cor}\label{maincor}
Let  $\gamma :\RZ\rightarrow \R^n$ be a closed simple arc-length parametrized curve  with $\E^\alpha (\gamma) < \infty$ for $2< \alpha<3$.
If $\gamma$ is a critical point of $\E^\alpha + \lambda \Ell$, then the curve $\gamma$ is analytic.
\end{cor}

Note that we study critical points of O'Hara's energy $\E^\alpha$ subject to a Lagrange multiplier length term since $\E^\alpha$ for $2<\alpha <3$ is not scale-invariant in contrast to the Möbius energy case.

\subsubsection*{\textbf{Exposé of the present work.}} 

The main goal of this article is to give a rigorous proof of Theorem~\ref{main}. We adapt the methods from the Möbius energy case in \cite{BV19} to generate a proof which is as elementary as possible. 
For the convenience of the reader we will provide detailed proofs in this article on the one hand to emphasize the differences to \cite{BV19} and, on the other hand, to make the article comprehensible without the need for a detailed reading of \cite{BV19}.

In Section~\ref{sec:preliminaries} we recall some basic definitions and properties of fractional Sobolev spaces and Fourier series. We also characterize analytic functions and state an ordinary differential equation (ODE)-version of the theorem of Cauchy-Kovalevskaya. We close the preliminaries with recapitulating Fa\`a di Bruno's formula. 
In Section \ref{section:decomposition} we decompose the first variation of O'Hara's energies into the orthogonal projection of a main part $Q^\alpha$ and two remaining parts $R_1^\alpha$ and $R_2^\alpha$ of lower order. The main part $Q^\alpha$ and its derivatives are estimated in Section \ref{sec:maintermQ}. Since the orthogonal projection of $Q^\alpha$ appears in the first variation of O'Hara's energies, we estimate the tangential part of $Q^\alpha$ using new estimates for a kind of fractional derivative of a product of two functions that, in Section \ref{sec:bilinearHilberttransform}, that behave like a fractional Leibniz rule. In Section \ref{sec:bilinearHilberttransform} and Section \ref{sec:remaining parts} we cut-off the singularities of the singular integrals involved and derive uniform estimates. More precisely, in Section \ref{sec:remaining parts} we rewrite  the orthogonal projection of the truncated remaining terms so that they may be expressed by integrals over analytic functions. 
In Section \ref{sec:mainpart} we realize that estimates do not depend on the cut-off parameter and hold for the orthogonal projections of $Q^\alpha$, $R_1^\alpha$, $R_2^\alpha$, and its derivatives. Finally,  we will use the  estimates to prove Theorem~\ref{main} using Cauchy's method of majorants.

\section{Preliminaries}\label{sec:preliminaries}

\subsection{Fractional Sobolev spaces}\label{sec:fractionalsobolevspaces}

We denote by $|\cdot | $ the Euclidean norm on $\C^n$ for any integer $n\geq 1$.
In this article we work with closed curves on the periodic domain $\RZ$ so 
that a curve $f: \R \rightarrow \R^n$ is periodic with unit periodicity. 

We know that for any $f\in L^2(\RZ, \C^n)$ the Fourier series of $f$ in $x\in \RZ$ is given by $\sum_{k\in\Z} \widehat{f}(k) e^{2\pi i k x},$ where the $k$-th Fourier coefficient of $f$ are given by 
\begin{align*}
\widehat{f}(k) = \int_0^1 f(x) e^{-2\pi i k x} \mathrm{d} x .
\end{align*}

The \emph{fractional Sobolev space} of order $s\geq 0$ (i.e.~the Bessel potential space of order $s \geq 0$) is defined as
\begin{align*}
H^s(\RZ, \C^n) := \{f\in L^2 (\RZ, \C^n) \ | \ \| f \|_{H^s}:= \| f \|_{H^s(\RZ,\C^n)} := \sqrt{(f,f)_{H^s}} < \infty  \},
\end{align*}
equipped with the scalar product  
\begin{align*}
(f,g)_{H^s} := (f,g)_{H^s(\RZ,\C^n)} := \sum_{k\in\Z} (1+k^2)^s \langle \widehat{f}(k), \widehat{g}(k) \rangle_{\C^n}.
\end{align*}
 
Furthermore, we will need the embeddings $H^s(\RZ,\C^n)\subseteq H^t(\RZ,\C^n)$ for any $0 \leq t<s$ (cf.~\cite[Prop.~2.1, Corol.~2.3]{DNPV12}) as well as $H^s(\RZ,\C^n) \subseteq C(\RZ,\C^n)$ for any $s> \frac{1}{2}$ (cf.~\cite[Theor.~8.2]{DNPV12}), from which we can deduce that the Fourier series of any $f\in H^s(\RZ,\C)$ with $s> \frac{1}{2}$ converges absolutely and uniformly to $f$. 
In fact the space $H^m$, for any integer $m\geq 0$, coincides with the classical Sobolev space and the norms $\|\cdot\|_{H^m}$ and $\|f\|_{W^m}:= (\sum_{\nu=0}^{m} \|\partial^\nu f\|^2_{L^2})^{\frac{1}{2}}$ are equivalent (cf.  \cite[Lem.~1.2]{R09} or \cite[7.62]{AF03}).
Furthermore, let us also  mention the Banach algebra property of $H^m$ for any integer $m\geq 1$, i.e. there exists a positive constant $C_m$ such that 
	\begin{align}\label{prop:Banachalgebra}
	\|fg\|_{H^m} \leq C_m \|f\|_{H^m} \|g\|_{H^m}
	\end{align}
for all $f,g \in H^m(\RZ,\R)$  (cf.~\cite[Theor.~4.39]{AF03}).

\subsection{Properties of analytic functions}

We shortly recall some basic characteristics of analytic functions. More information on analytic functions can be found for instance in  \cite[Chap.~6.4]{E98} and \cite[Chap.~1,~D]{F95}. 

It is well-known that analytic functions can be characterized by \cite[Prop.~2.2.10]{KP02} and a standard covering argument as follows:
\begin{thm}\label{thm:Tayloranalytic}
A function 	 $f\in C^\infty (\Omega,\R^n)$ on $\Omega \subseteq \R^m$ open, $n,m \geq 1$,
	is analytic on $\Omega$ if and only if for every compact set $K \subseteq \Omega$
	there are positive constants $r_K$ and $C_K$ such that
	\begin{align*}
	\|\partial^\alpha f\|_{L^\infty(K)} \leq C_K \frac{|\alpha|!}{r_K^{|\alpha|}}
	\end{align*}
	holds for every multiindex $\alpha \in \mathbb N_0^m$.
\end{thm} 

The previous theorem, together with the embedding of the classical Sobolev space $W^{k}=W^{k,2}$ into $C^0$ for $k\in \mathbb N$ with $k>\frac m2$, yields the following:

\begin{cor} 
A function 	 $f\in C^\infty (\Omega,\R^n)$  on $\Omega \subseteq \R^m$ open, $n,m \geq 1$,	
	is analytic on $\Omega$ if for every compact set $K \subset \Omega$
	there are positive constants $r_K$ and $C_K$ such that
	\begin{align*}
	\|\partial^\alpha f\|_{W^1(K)} \leq C_K \frac{|\alpha|!}{r_K^{|\alpha|}}
	\end{align*}
	holds for every multiindex $\alpha \in \mathbb N_0^m.$
\end{cor}

By the equivalence of the $W^1$- and the $H^1$-norm on $\mathbb R / \mathbb Z$ as well as the embedding $H^s\subseteq H^t$ for any $t < s$ we obtain:

\begin{cor} \label{cor:TayloranalyticSob}
	Let $f\in C^\infty (\mathbb R / \mathbb Z,\R^n)$ and $s\in (0,1)$. Then the function $f$ is analytic on $\RZ$ if  there are positive constants $r$ and $C$ such that
	\begin{align*}
	\|\partial^k f\|_{H^{1+s}} \leq C \frac{k!}{r^{k}}
	\end{align*}
	holds for all integers $k \geq 0$.
\end{cor}

Furthermore, we will need a ODE version of the theorem of Cauchy-Kovalevskaya, which is originally an existence and uniqueness theorem for analytic nonlinear partial differential equations associated with Cauchy initial value problems.

\begin{thm}[Cauchy-Kovalevskaya - ODE case]\label{CK}
	Suppose the function $g: \R^n \rightarrow \R^n$ is real analytic around $0$ and the function $f \in C^\infty((-\varepsilon, \varepsilon),\R^n)$ of the form $f(x)= (f_1(x),\ldots,f_n(x))$, for any $ \varepsilon > 0$, is a solution of the initial value problem
	\begin{align*}
	\dot{f}(x) & = g (f(x)) \textnormal{ for } x \in (-\varepsilon,\varepsilon) \textnormal{ with}\\
	f(0)& =0.
	\end{align*}
	Then the function $f$ is real analytic around $0$.
\end{thm}

One possible method to prove the theorem of Cauchy-Kovalevskaya is the \emph{method of majorants} (c.f. \cite[Chap.~1,~D]{F95} or \cite[Chap.~4.6.3, Theor.~2]{E98}). This method turns out to be useful in proving the main result of this article.

\subsection{Faà di Bruno's formula}

	The $k$-th derivative of the composition of two functions $f, g\in C^k (\R,\R)$ can be expressed by Faà di Bruno's formula which is given by
	\begin{align*}
	&\left(\frac{d}{dt}\right)^k g(f(t)) = \\
	&\sum_{\substack{m_1+2m_2+ \cdots km_k=k, \\ m_1,\ldots,m_k\in\N_0}} \frac{k!}{m_1!1!^{m_1}m_2!2!^{m_2}\ldots m_k!k!^{m_k}} g^{(m_1+\cdots +m_k)}(f(t))\prod_{j=1}^{k} (f^{(j)} (t))^{m_j}.
	\end{align*}
	This formula can be generalized to the multivariate case by a result of  \cite{M00}. In our case the precise generalized Faà di Bruno's formula is not required.
	Nonetheless  we will make use of the the following that can be easily proven from scratch by induction:

	For any functions $f \in C ^k (\mathbb R, \mathbb R ^n)$ and $g \in C^k(\mathbb R ^n, \mathbb R)$, for integers $n\geq 1$ and $k\geq 0$, there exists a universial polynomial $p_k^{(n)}$ with non-negative coefficients, which are independent of $f$ and $g$, such that 
	\begin{equation}\label{eqn:FdB}
	\partial^k g(f(x)) = p_k^{(n)} (\{\partial^\alpha g\}_{|\alpha|\leq k}, \{\partial^j f_i\}_{i=1, \ldots, n, j=1, \ldots, k} ).
	\end{equation}
	In addition, $p_k^{(n)}$ is one-homogeneous in the first entries.

\section{Decomposition of the first variation of $\E^\alpha$}\label{section:decomposition}

Results concerning the differentiability of O'Hara's   knot energies   go back to the paper \cite{FHW94} where the Gâteaux differentiability and the $L^2$-gradient of the Möbius energy $\E^{2}$ was derived. 
Subsequently, a linearized version for the gradient of the Möbius energy was given by He \cite[Lem.~2.2]{H00} in the context of the heat flow.
An application of He's linearization trick to a range of O'Hara's knot energies $\E^\alpha$, for $2 \leq \alpha <3$, was given by Reiter \cite[Chap.~2]{R12}.

In the following we will use the linearization trick of He and Reiter to decompose the first variation of the energies $\E^\alpha$, for $2 < \alpha < 3$, into a highest order quasilinear part $Q^\alpha$ and two remaining  $R_1^\alpha$, $R_2^\alpha$ parts of lower order (as it was done in \cite{BR13} and \cite[Theor.~2.3]{B18}).
A similar decomposition of the first variation proved apt for the analyticity proof of critical points of the Möbius energy that was given in  \cite{BV19}.

Let $\gamma \in C^\infty(\RZ,\R^n)$ be a  simple closed arc-length parametrized curve  
and $x\in\RZ$. We  recall that the orthogonal projection  $P^\perp_{\dot{\gamma}(x)} : \R^n \rightarrow \R^n$ at the point $\gamma(x)$ onto the normal space of the curve $\gamma$  is given by 
\begin{align*}
  (P^\perp_{\dot{\gamma}}v)(x) = P^\perp_{ \dot{\gamma}(x)}(v) = v-\langle v,\dot{\gamma}(x) \rangle_{\R^n} \dot{\gamma}(x)
\end{align*}
for any $v\in \R^n$. 
Furthermore, the tangential projection  $P^T_{\dot{\gamma}} : \R^n \rightarrow \R^n$  at the point $\gamma(x)$ onto the tangent space of the curve $\gamma$  is given by 
\begin{align*}
(P^T_{\dot{\gamma}}v)(x) = P^T_{\dot{\gamma}(x)} (v) = \langle v, \dot{\gamma}(x) \rangle_{\R^n} \dot{\gamma}(x)
\end{align*} 
for any $v\in \R^n$. We remark also that the curvature vector of the curve $\gamma$ is denoted by $\kappa$
i.e.~$\kappa = (\frac{d}{ds})^2 \gamma$ where $\frac{d}{ds}$ is the derivative with respect to the arc-length. 
We recall the following:

\begin{thm}[{Reiter \cite[Theor.~2.24]{R12}}]\label{thm:1stvar}
The first variation of $\E^{\alpha}$, for $2 < \alpha < 3$, at a simple regular curve $\gamma \in C^\infty(\RZ,\R^n)$ in direction $h \in H^2 (\RZ,\R^n)$ can be expressed as
	\begin{align*}
	\delta \E^\alpha(\gamma, h) = \int_{0}^{1} \langle (H^\alpha\gamma)(x), h(x) \rangle_{\R^n} |\dot{\gamma}(x)|  \mathrm{d} x,
	\end{align*}
	where
	\begin{align*}
	&(H^\alpha \gamma)(x) := \lim_{\epsilon \downarrow 0} \int_{|w| \in [\epsilon, \frac{1}{2}]} P^\perp_{\dot{\gamma}(x)} \{ 2\alpha \frac{\gamma(x+w)-\gamma(x)}{|\gamma(x+w)-\gamma(x)|^{2+\alpha}} \\
	&\hspace{5em}  -(\alpha -2) \frac{\kappa(x)}{\D(\gamma(x+w),\gamma(x))^{\alpha}} -2 \frac{\kappa(x)}{|\gamma(x+w)-\gamma(x)|^{\alpha}} \} |\dot{\gamma}(x+w)| \mathrm{d} w.
	\end{align*}
\end{thm}

Due to the  the fact that 
$P^\perp_{\dot{\gamma}(x)}(\dot{\gamma}(x)) = 0$, 
$\langle \dot{\gamma}(x),\ddot{\gamma}(x) \rangle_{\R^n} = 0$,
 the curve $\gamma$ is parametrized by arc-length
  and the orthogonal projection $P^\perp_{\dot{\gamma}(x)}$ is linear, 
one can easily rewrite $H^{\alpha} \gamma$ as 
\begin{align}\label{formula:H}
(H^\alpha\gamma) (x) = (P^\perp_{\dot{\gamma}}\widetilde{H}^\alpha \gamma) (x) = P^\perp_{\dot{\gamma}(x)} ((\widetilde{H}^\alpha \gamma) (x)) ,
\end{align}
where 
\begin{align*}
(\widetilde{H}^{\alpha} \gamma) (x) := \lim_{\varepsilon \downarrow 0} & \int_{|w| \in [\varepsilon,\frac{1}{2}]} \{ 2\alpha \frac{\gamma(x+w)-\gamma(x)-w\dot{\gamma}(x)}{|\gamma(x+w)-\gamma(x)|^{2+\alpha}} \\
& -(\alpha -2) \frac{\ddot{\gamma}(x)}{ \D(\gamma(x+w),\gamma(x))^{\alpha}}  -2 \frac{\ddot{\gamma}(x)}{ |\gamma(x+w)-\gamma(x)|^{\alpha}} \}  \mathrm{d} w .
\end{align*}

Now we want to apply the  strategy of  He and Reiter indicated as in the above.  We decompose $\widetilde{H}^\alpha \gamma$ pointwise for all $x\in \RZ$ into three functionals by
\begin{align}\label{formula:Htilde}
\widetilde{H}^\alpha\gamma = \alpha Q^\alpha\gamma + 2\alpha R_1^\alpha \gamma - 2 R_2^\alpha \gamma  ,
\end{align}
where
\begin{align*}
	Q^\alpha\gamma (x) 
	&= \lim_{\varepsilon \downarrow 0}Q^{\alpha,\varepsilon}\gamma(x) \\	
	R_1^\alpha \gamma (x) &= \lim_{\varepsilon \downarrow 0}R_1^{\alpha,\varepsilon}\gamma(x) \\
	R_2^\alpha \gamma (x) &= \lim_{\varepsilon \downarrow 0}R_2^{\alpha,\varepsilon}\gamma(x) 
\end{align*}
are given, for any $0<\varepsilon\leq \frac 12$,  by  
\begin{align*}
(Q^{\alpha,\varepsilon}\gamma)(x) &= \int_{|w| \in [\varepsilon, \frac{1}{2}]} \left(2 \frac{\gamma(x+w)-\gamma(x)-w\dot{\gamma}(x)}{w^2} - \ddot{\gamma}(x) \right) \frac{\mathrm{d} w}{|w|^\alpha}, \\
(R_1^{\alpha,\varepsilon} \gamma)(x) &= \int_{|w| \in [\varepsilon, \frac{1}{2}]} (\gamma(x+w)-\gamma(x)-w\dot{\gamma}(x)) \left(\frac{1}{|\gamma(x+w)-\gamma(x)|^{\alpha +2}} - \frac{1}{ |w|^{\alpha +2}}\right)\mathrm{d} w, \\
(R_2^{\alpha,\varepsilon} \gamma)(x) &= \int_{|w| \in [\varepsilon, \frac{1}{2}]} \ddot{\gamma}(x)\left(\frac{1}{|\gamma(x+w)-\gamma(x)|^{\alpha}} - \frac{1}{ |w|^{\alpha }}\right)\mathrm{d} w
\end{align*}
Blatt and Reiter  \cite{BR13} have already used this partitioning for studying the regularity of stationary points of O'Hara's energies $\E^{\alpha}$, $2<\alpha <3$, and furthermore, Blatt (cf.~\cite[Theor.~2.3]{B18}) used it in the study of the gradient flow for the same range of O'Hara's knot energies.

In the next section we will see that $Q^\alpha$ contains the highest order part of the first variation of $\E^\alpha$. Thereafter the challenge will be to attain sufficient estimates of the tangential part of $Q^\alpha$, whereas $R_1^\alpha$ and $R_2^\alpha$ are of lower order and easier to get under control.

\section{An estimate for the $Q^\alpha$ term}\label{sec:maintermQ}

Let the quantities 
\begin{align*}
q_k^{(\alpha)} := \begin{cases}
\frac{4\pi}{|k|^{\alpha +1}} \int_0^\pi \frac{(k\tau)^2 - 2 + 2\cos(k\tau)}{\tau^{2+\alpha}} \mathrm{d}\tau & \quad k\neq 0, \\
0 & \quad k = 0
\end{cases}
\end{align*}
and  set 
\begin{align*}
	\lambda_k^{(\alpha)} := \int_0^{k\pi} \frac{\sin (\tau)}{\tau^{\alpha - 1}} \mathrm{d} \tau.
\end{align*}
We remark that the $\lambda_k$ are well-defined for all $k\in\Z$ and $\lambda_\infty^{(\alpha)}  := \lim_{k\rightarrow \infty} \lambda_k^{(\alpha)} < \infty$ (cf.~\cite[Lem.~2.4]{R12}). 
We deduce from \cite[Prop.~2.3]{R12}, by switching from the weak formulation of the Euler-Lagrange equation to the strong one, the following: 

\begin{thm}\label{thm:QFourier}
	For every curve $\gamma \in C^\infty(\RZ, \R^n)$ the term $Q^\alpha \gamma$ is a $C^\infty$-function and its Fourier coefficients, for all $k\in \Z$, are given by
	\begin{align}\label{eq:fouriercoeffQ}
	\widehat{Q^{\alpha} \gamma}(k) = -q_k^{(\alpha)} |k|^{\alpha +1} \hat{\gamma} (k) 
	.\end{align}
	The constants $q_k^{(\alpha)}$ are bounded and satisfy $$0 \leq q_k^{(\alpha)} = \frac{8\pi \lambda_k^{(\alpha)}}{\alpha (\alpha + 1 )(\alpha -1)} + O(\tfrac{1}{k^{1-\alpha}})$$ as $k \rightarrow \infty$.
\end{thm}
\begin{proof}
By Taylor's expansion up to third order we obtain 
\begin{align}
Q^\alpha \gamma (x) & = \lim_{\varepsilon\downarrow 0} 2  \int_{|w| \in [\varepsilon, \frac{1}{2}]} \frac{w }{|w|^\alpha}   \int_0^1 (1-t)^2 \dddot{\gamma}(x+tw) \mathrm{d}t \mathrm{d} w \nonumber \\
& = \lim_{\varepsilon\downarrow 0} 2  \int_{|w| \in [\varepsilon, \frac{1}{2}]} \frac{w }{|w|^\alpha}   \int_0^1 (1-t)^2 \left(\dddot{\gamma}(x+tw) - \dddot{\gamma}(x)\right) \mathrm{d}t  \mathrm{d} w , \label{eq:Qcont}
\end{align}
from which directly follows that $Q^\alpha$ maps $C^{3,\beta}$ to $L^\infty$ for any $0 < \beta \leq 1$. From the linearity of $Q^\alpha$ we get $\partial^l Q^\alpha \gamma (x) = Q^\alpha\partial^l \gamma(x)$, so $Q^\alpha$ maps $C^{l+3,\beta}$ to $C^{l-1,1}$ for all integers $l\geq 1$, which demonstrates the first part of the Theorem's statement.
 
Now we define the bilinear functional
\begin{align*}
\widetilde Q^\alpha (\gamma,\eta) := \lim_{\varepsilon  \downarrow 0 } \int_{\mathbb R / \mathbb Z} \int_{|w| \in [\varepsilon, \frac 1 2]} \Big( \frac {\langle \gamma(x+w)-\gamma(x),\eta(x+w)-\eta(x)\rangle_{\R^n} }{w^2}- \langle \dot \gamma (x) , \dot \eta (x) \rangle_{\R^n} \Big) \frac{\mathrm{d}w \mathrm{d}x}{|w|^\alpha} .
\end{align*}
for any $\gamma,\eta \in H^{\frac{\alpha+1}{2}}(\mathbb R / \mathbb Z, \mathbb R ^n)$. Applying continuous and discrete integration by parts, gives
\begin{align}
\widetilde Q^\alpha (\gamma,\eta) & = \lim_{\varepsilon  \downarrow 0 } \int_{\mathbb R / \mathbb Z} \int_{|w| \in [\varepsilon, \frac 1 2]} \left( -2 \frac {\langle \gamma(x+w)-\gamma(x),\eta(x)\rangle_{\R^n} }{w^2} - \langle \dot \gamma(x) , \dot \eta(x) \rangle_{\R^n} \right) \frac{\mathrm{d}w \mathrm{d}x}{|w|^\alpha}  \nonumber \\
& = -\lim_{\varepsilon \downarrow 0} \int_{\mathbb R / \mathbb Z} \langle Q^{\alpha,\varepsilon} \gamma(x), \eta(x) \rangle_{\R^n} \mathrm{d}x 
= - \int_{\mathbb R / \mathbb Z} \langle Q^\alpha \gamma(x) ,\eta(x) \rangle_{\R^n} \mathrm{d}x   \label{formula:Qtildecalc}
\end{align}
because $Q^{\alpha,\varepsilon} \gamma$ converges to $Q^\alpha \gamma$ in $L^\infty(\RZ,\R^n)$ as $\epsilon\downarrow 0$ by \eqref{eq:Qcont}. 
Furthermore, by \cite[Prop.~2.3]{R12}, we have that
\begin{align*}
 \widetilde Q^\alpha(\gamma,\eta) = \sum_{k \in \mathbb Z} q_k^{(\alpha)} |k|^{\alpha+1} \hat{\gamma} (k) \hat {\eta}(k)
\end{align*}
and by Plancherel's identity we get 
\begin{align*}
  \int_{\mathbb R / \mathbb Z} \langle Q^\alpha\gamma ,\eta \rangle dx  = \sum_{k \in \mathbb Z} \widehat{Q^\alpha \gamma}(k) \widehat{\eta}(k).
\end{align*}
So by  comparing the Fourier coefficients in \eqref{formula:Qtildecalc} we have
$\widehat{Q^\alpha f}(k) = -q_k^{(\alpha)}|k|^{\alpha +1} \hat{\gamma} (k) $ for all $ k\in \Z.$ The properties of the coefficients $q_k^{(\alpha)}$ directly follow from the proof of \cite[Prop.~2.3]{R12}.
\end{proof}		

By Theorem~\ref{thm:QFourier} we obtain the following essential corollary.

\begin{cor}  \label{kor:Ql}
	There exists a positive constant $ \widetilde{C}$ such that for all curves $\gamma \in C^\infty(\RZ, \R^n)$ and  integers $l\geq 0$ 
	\begin{align*}
	\|\partial^{l+3}\gamma \|_{H^{m+\alpha-2}} \leq \widetilde{C} \|\partial^{l} Q^{\alpha}\gamma\|_{H^m}
	\end{align*}
	holds for any real numbers $m\geq 0$ and $2 < \alpha < 3$.
\end{cor}
\begin{proof}
	By considering Theorem~\ref{thm:QFourier}, in particular the form of the Fourier coefficients \eqref{eq:fouriercoeffQ}, the definition of the fractional Sobolev-norm, elementary properties of Fourier coefficients and the simple estimate $(1+k^2)^{\alpha -2} \leq (2|k|)^{2(\alpha -2)}$ for any $k\in \Z  \setminus 0$,  we get 
	\begin{align*}
	\|\partial^{l}Q^\alpha\gamma \|_{H^m}^2 &= \sum_{k\in \Z} (1+|k|^2)^m|\widehat{\partial^{l}Q^\alpha \gamma}(k)|^2 \\
	&= \sum_{k\in \Z} (1+|k|^2)^m (2\pi| k|)^{2l} |\widehat{Q^\alpha \gamma}(k)|^2 \\
	&= \sum_{k\in \Z} (1+|k|^2)^m (2\pi| k|)^{2l} (q_k^{(\alpha)})^2 |k|^{2(\alpha +1)} |\hat{\gamma} (k)|^2\\
	&\geq \inf_{k\in\Z\setminus 0} \{(q_k^{(\alpha)})^2 (2\pi)^{-6} 2^{2-\alpha} \} \sum_{k \in \Z} (1+|k|^2)^{m+\alpha-2} (2\pi| k|)^{2(l+3)}   |\hat{\gamma} (k)|^2 \\
	&= \widetilde{C}^{-2} \sum_{k \in \Z} (1+|k|^2)^{m+\alpha-2} |\widehat{\partial^{(l+3)}\gamma} (k)|^2 \\
	&= \widetilde{C}^{-2}   \|\partial^{(l+3)}\gamma \|_{H^{m+\alpha-2}}^2,
	\end{align*}
	where $\widetilde{C}:= \inf_{k\in\Z\setminus{0}} \{(q_k^{(\alpha)})^2 (2\pi)^{-6} 2^{2-\alpha}\}^{-\frac{1}{2}}$ is a positive constant.
\end{proof}

\section{A fractional leibniz rule and the form of $P^T_{\dot{\gamma}} Q^\alpha\gamma$} \label{sec:bilinearHilberttransform}

Let $\gamma \in C^\infty(\RZ,\R^n)$ be a closed simple curve parametrized by arc-length. 
Recall that in the decomposed first variation of O'Hara's range of energies $\E^\alpha$, for $2 < \alpha < 3$, there appears the orthogonal projection of the main part $Q^\alpha \gamma$, which is given by $P^\perp_{\dot{\gamma}} Q^\alpha\gamma =  Q^\alpha\gamma - P^T_{\dot{\gamma}} Q^\alpha\gamma$. Since we have worked out an estimate for $Q^\alpha \gamma$ in the previous section, it remains to study $P^T_{\dot{\gamma}} Q^\alpha\gamma$, i.e.~the tangential part of $Q^\alpha\gamma$.

In this section we will see that a type of fractional derivative of a product of two functions can help to estimate the tangential part of $Q^\alpha\gamma$. We interpret the resulting estimate as a \emph{fractional Leibniz rule}. 
 
In order to avoid problems coming from the singularities of the integrand, we will work with the truncated functional $Q^{\alpha,\varepsilon}$ for $0 < \varepsilon \leq \frac 12$. By using Taylor's approximation up to second order with remainder term in integral from and 
$\left< \dot{\gamma}(x), \ddot{\gamma}(x) \right> = 0$ for all $x\in \RZ$ together with the bilinearity of the scalar product and $\alpha > 2$, we can write
\begin{align}\label{formula:PTQepsilon}
& \langle Q^{\alpha,\varepsilon}\gamma(x), \dot{\gamma}(x) \rangle_{\R^n} \nonumber\\
& = \int_{|w|\in [\varepsilon, \frac{1}{2}]} \left< 2 \frac{\gamma(x+w)-\gamma(x)-w\dot{\gamma}(x)}{w^2} - \ddot{\gamma}(x) , \dot{\gamma}(x) \right>_{\R^n}  \frac{\mathrm{d}w}{|w|^\alpha} \nonumber\\
& = 2  \int_{|w|\in [\varepsilon, \frac{1}{2}]} \int_0^1 (1-t) \left< \ddot{\gamma}(x+tw) , \dot{\gamma}(x) \right>_{\R^n}  \mathrm{d}t \frac{\mathrm{d}w}{|w|^{\alpha}} \nonumber\\
& = 2  \int_{|w|\in [\varepsilon, \frac{1}{2}]} \int_0^1 (1-t)  \left<\ddot{\gamma}(x+tw), \frac{\dot{\gamma}(x) - \dot{\gamma}(x+tw)}{w}\right>_{\R^n}  \mathrm{d}t \frac{\mathrm{d}w}{w|w|^{\alpha-2}}  \nonumber\\
& = 2 \int_{|w|\in [\varepsilon, \frac{1}{2}]} \int_0^1  \int_0^1 (1-t) (-t) \frac{\left<\ddot{\gamma}(x+tw), \ddot{\gamma}(x+stw)\right>_{\R^n}}{w}  \mathrm{d}s \mathrm{d}t \frac{\mathrm{d}w}{|w|^{\alpha-2}}.
\end{align} 
With $\varepsilon\downarrow 0$ we also get
\begin{align}  
&(P^T_{\dot{\gamma}} Q\gamma)(x) \nonumber\\
& = 4 \lim_{\varepsilon \downarrow 0} \int_{|w|\in [\varepsilon, \frac{1}{2}]} \iint_{[0,1 ]^2} (1-t) (-t) \frac{\left<\ddot{\gamma}(x+tw), \ddot{\gamma}(x+stw)\right>_{\R^n}}{w|w|^{\alpha-2}} \dot{\gamma}(x) \mathrm{d}s \mathrm{d}t \mathrm{d}w. \label{formula:compTQ}
\end{align}
The last terms of \eqref{formula:PTQepsilon} and  \eqref{formula:compTQ} motivate us to introduce the following:

\begin{defn}
	Let $s_1,s_2 \in [0,1]$, $0 < \varepsilon \leq \frac{1}{2} $ and $\beta \in (0,1)$. 
	Then the bilinear singular integral
	 $H_{s_1,s_2,\beta}: C^1 (\RZ,\R) \times C^1 (\RZ,\R) \rightarrow L^\infty (\RZ,\R)$ is given by 
	\begin{align}\label{defn:Hs1s2}
	H_{s_1,s_2,\beta}(f,g) (x) := \lim_{\varepsilon \downarrow 0} \int_{|w| \in [\varepsilon, \frac 12]} \frac{f(x+s_1w)  g(x+s_2w)}{w |w|^\beta} \mathrm{d}w 
	\end{align}
	for all $x\in \RZ$ and, its truncated version
	$H_{s_1,s_2,\beta}^\varepsilon: C^1 (\RZ,\R) \times C^1 (\RZ,\R) \rightarrow L^\infty (\RZ,\R)$ is given by 
	\begin{align}\label{defn:Hs1s2eps}
		& H_{s_1,s_2,\beta}^\varepsilon(f,g) (x) := \int_{|w| \in [\varepsilon, \frac 12]} \frac{f(x+s_1w)  g(x+s_2w)}{w |w|^\beta} \mathrm{d}w
	\end{align} 
	for all $x\in\RZ$.
\end{defn}

If $\beta$ was allowed to vanish, the previous definition would yield the bilinear Hilbert transform. Since we have the factor $|w|^\beta$ for $\beta \in (0,1)$ in the denominator in the definition of the bilinear singular integral \eqref{defn:Hs1s2}, the formula is reminiscent  of a fractional derivative.
 
\begin{lem}\label{lem:Hs1s2cont}
For every $s_1,s_2 \in [0,1]$, $0 < \varepsilon \leq \frac 12$ and  $\beta \in (0,1)$ 
the transforms 	$H_{s_1,s_2,\beta}$ and $H_{s_1,s_2,\beta}^\varepsilon$
are well-defined, continuous and bilinear from $C^1 (\RZ,\R)\times C^1 (\RZ,\R)$  to $L^\infty (\RZ,\R)$.
Furthermore, the truncated bilinear integral $H_{s_1,s_2,\beta}^\varepsilon$ is also continuous from $C (\RZ,\R)\times C (\RZ,\R)$  to $L^\infty (\RZ,\R)$.
\end{lem}
\begin{proof}
	It is easy to see by the linearity of the integral in \eqref{defn:Hs1s2} that $H_{s_1,s_2,\beta}$ is indeed linear in both components. The function $H_{s_1,s_2,\beta}$ is also well-defined for every $f,g \in C^1 (\RZ,\R)$, which we can see by adding a zero to the definition of the bilinear singular integral \eqref{defn:Hs1s2} as
	\begin{align*}
	H_{s_1,s_2,\beta}(f,g)(x)= \lim_{\varepsilon \downarrow 0} \int_{|w|\in [\varepsilon, \frac{1}{2}]} \frac{f(x+s_1w)  g(x+s_2w)}{w|w|^\beta} - \frac{f(x)  g(x)}{w|w|^\beta}  \mathrm{d}w,
	\end{align*}
	by inserting a second zero and by the Lipschitz-continuity of continuously differentiable functions on a compact set such that
	\begin{align}\label{lem:Hs1s2}
	& \left|H_{s_1,s_2,\beta}(f,g)(x)\right| \nonumber\\
	& \leq \lim_{\varepsilon \downarrow 0} \int_{|w|\in [\varepsilon, \frac{1}{2}]} \frac{|f(x+s_1w)  g(x+s_2w) - f(x)g(x+s_2w) +f(x)g(x+s_2w) -f(x)  g(x)|}{|w|^{1+\beta}}  \mathrm{d}w \nonumber\\
	& \leq \lim_{\varepsilon \downarrow 0} \int_{|w|\in [\varepsilon, \frac{1}{2}]} \frac{|f(x+s_1w) - f(x)| | g(x+s_2w)| +  |f(x)| |g(x+s_2w)- g(x)|}{|w|^{1+\beta}}  \mathrm{d}w \nonumber\\
	& \leq \lim_{\varepsilon \downarrow 0} \int_{|w|\in [\varepsilon, \frac{1}{2}]} \frac{\|f'\|_\infty s_1|w| \|g\|_\infty +  \|f\|_\infty \|g'\|_\infty s_2 |w|}{|w|^{1+\beta}}  \mathrm{d}w \nonumber\\
	& \leq C (s_1,s_2,\beta)  \|f\|_{C^1} \|g\|_{C^1}  < \infty
	\end{align}
	for some positive finite constant $C(s_1,s_2,\beta)$. Hence we can deduce from \eqref{lem:Hs1s2} that $H_{s_1,s_2,\beta}$ is continuous. 
	We get the same properties for $H_{s_1,s_2,\beta}^\varepsilon$ for any 	$0<\varepsilon\leq \frac 12$
	by transfering the previous arguments. Additionally, since we cut off the singularity in the definition of $H_{s_1,s_2,\beta}^\varepsilon$, the continuity of  $H_{s_1,s_2,\beta}^\varepsilon$ from $C (\RZ,\R)\times C (\RZ,\R)$  to $L^\infty (\RZ,\R)$ follows by the following elementary estimates 
	\begin{align*}
	 \left|H_{s_1,s_2,\beta}^\varepsilon(f,g)(x)\right| & = \left| \int_{|w|\in [\varepsilon, \frac{1}{2}]} \frac{f(x+s_1w)  g(x+s_2w)}{w|w|^\beta} - \frac{f(x)  g(x)}{w|w|^\beta}  \mathrm{d}w \right| \\
	& \leq \int_{|w|\in [\varepsilon, \frac{1}{2}]} \frac{|f(x+s_1w)| | g(x+s_2w)| +  |f(x)| |g(x)|}{|w|^{1+\beta }}  \mathrm{d}w \\
	& \leq  \int_{|w|\in [\varepsilon, \frac{1}{2}]} \frac{\|f\|_\infty \|g\|_\infty +  \|f\|_\infty \|g\|_\infty }{|w|^{1+\beta}}  \mathrm{d}w \\
	& \leq C (\beta, \varepsilon)  \|f\|_{L^\infty} \|g\|_{L^\infty}  < \infty
	\end{align*}
	for all $x\in [0,1]$ and some constant $0 < C(\beta, \varepsilon) < \infty$. 
\end{proof}

We next estimate the truncated bilinear integral \eqref{defn:Hs1s2eps} 
that leads to a new kind of fractional Leibniz rule.

\begin{thm}\label{thm:BilinearHT}
	Let $m> \frac{1}{2}$, $0 < \varepsilon \leq \frac{1}{2} $, $s_1,s_2 \in [0,1]$ and $\beta \in (0,1)$. Then there exists a positive constant $C_H = C_H(m, \beta) < \infty$ independent of $\varepsilon$ such that 
	\begin{align}\label{formula:BilinearHT}
	\|H_{s_1,s_2,\beta}^\varepsilon(f,g)\|_{H^m} \leq C_H \|f\|_{H^{m+ \beta}} \|g\|_{H^{m+ \beta }}
	\end{align}
	for all 
	$f,g \in  C^\infty(\RZ,\R^n)$.
\end{thm}

For the proof we will identify, for any $p \in [1,\infty)$, the sequence space $\ell^p$ by
\begin{align*}
\ell^p := \{x = (x_k)_{k\in\Z} \in \C^\Z \ | \ \|x\|_{\ell^p}:= (\textstyle \sum_{k\in\Z} |x_k|^p)^{\frac{1}{p}} < \infty \}
\end{align*}
and use the following lemmata.

\begin{lem}[{Young's inequality}]\label{lem:Young}
	Let $p,q,r \geq 1$ such that $\frac{1}{p} + \frac{1}{q}=\frac{1}{r}+1$. For sequences $x\in \ell^p$ and $y\in  \ell^q$ 
the convolution 
 $(x*y)(k):= \sum_{n\in \Z} x(n) y(k-n)$ is such that  $x*y \in \ell^r$ and
	\begin{align*}
	\|x*y\|_{\ell^r} \leq \|x\|_{\ell^p} \|y\|_{\ell^q} .
	\end{align*}
\end{lem}

\begin{lem}[{Sobolev-type inequality \cite[Lem.~5.5]{BV19}}]\label{lem:Sobolev}
	Let $f\in H^m(\RZ, \R^n)$ for $m> \frac{1}{2}$. Then $\hat{f}:= (\hat{f}(k))_{k\in\Z} \in \ell^1$ and there exists a positive constant $C_0 < \infty$ such that
	\begin{align*}
	\|\hat{f}\|_{\ell^1} \leq C_0 \|f\|_{H^m}.
	\end{align*}
\end{lem}

\begin{proof}[Proof of Theorem~\ref{thm:BilinearHT}] 
	Let $f,g \in  C^\infty(\RZ,\R)\subseteq H^m(\RZ,\R)$. Because of $m > \frac{1}{2}$ and \cite[Theor.~8.2]{DNPV12}, the partial sums of the Fourier series, i.e. $	p_n (x) := \sum_{k=-n}^{n} \hat{f}(k) e^{-2\pi i k x}$ and $q_n (x):=\sum_{k=-n}^{n} \hat{g}(k) e^{-2\pi i k x} $ for any $n\in\N_0$ and $x\in \RZ$, converge uniformly to the functions $f$ and $g$, that means
	\begin{align}\label{formula:apprfg}
	\begin{split}
	\|f-p_n \|_\infty  \rightarrow 0 \textnormal{ and } \|g-q_n \|_\infty \rightarrow 0 \textnormal{ as } n\rightarrow \infty.
	\end{split}		
	\end{align}
	Note that since $f,g \in  C^\infty(\RZ,\R)$, also $p_n, q_n \in C^\infty (\RZ, \R)$ holds.

	Our approach is to firstly prove the estimate \eqref{formula:BilinearHT} for the approximating functions $p_n$ and $q_n$ and secondly derive the actual statement of Theorem~\ref{thm:BilinearHT} by passing to the limit $n\rightarrow \infty$.
	We start by interchanging the integrals 
	two times due to $\varepsilon > 0$ and taking account of the fact that the Fourier coefficients of a product of two functions are the convolution of their Fourier coefficients, to gain
	\begin{align}\label{formula:pnqnepsilonI}
	& \widehat{H_{s_1,s_2,\beta}^\varepsilon(p_n,q_n)}(k) \nonumber\\
	 & = \int_0^1 \int_{|w|\in [\varepsilon, \frac{1}{2}]} \frac{p_n(x+s_1w)  q_n(x+s_2w)}{w|w|^\beta} e^{-2\pi i k x} \mathrm{d}w \mathrm{d}x  \nonumber\\
	& = \int_{|w|\in [\varepsilon, \frac{1}{2}]} \int_0^1  p_n(x+s_1w)  q_n(x+s_2w) e^{-2\pi i k x}  \mathrm{d}x \frac{\mathrm{d}w }{w|w|^\beta} \nonumber \\
	& = \int_{|w|\in [\varepsilon, \frac{1}{2}]}  \sum_{l\in\Z} \widehat{p_n}(l) e^{2\pi ils_1 w} \widehat{q_n}(k-l) e^{2\pi i(k-l)s_2 w} \frac{\mathrm{d}w }{w|w|^\beta} \nonumber \\ 
	& = 
	\begin{cases}
	0       & \quad \text{if } |k| > 2n \\
	\int_{|w|\in [\varepsilon, \frac{1}{2}]}  \sum_{l=-n}^{n} e^{2\pi i(ls_1+ (k-l)s_2) w} \widehat{f}(l) \widehat{g}(k-l) \frac{\mathrm{d}w }{w|w|^\beta}  & \quad \text{if } |k| \leq 2n
	\end{cases} 
	.\end{align}
	Now we focus on the non-trivial Fourier coeffients of $H_{s_1,s_2,\beta}^\varepsilon(p_n,q_n)$, that means on the indices $|k|\leq 2n$. By defining $\phi_{l,k} := 2\pi (ls_1 +(k-l)s_2) \in \R$, substitution and interchanging integral and sum, we get 
	\begin{align}\label{formula:pnqnepsilonII}
	& \int_{|w|\in [\varepsilon, \frac{1}{2}]}  \sum_{l=-n}^{n} e^{iw\phi_{l,k} } \widehat{f}(l) \widehat{g}(k-l) \frac{\mathrm{d}w }{w|w|^\beta} \nonumber\\ 
	&\qquad \qquad= \int_\varepsilon^{\frac{1}{2}}  \sum_{l=-n}^{n} \frac{e^{i w\phi_{l,k}} - e^{-i w \phi_{l,k}}}{w|w|^\beta} \widehat{f}(l) \widehat{g}(k-l) \mathrm{d}w \nonumber\\
	&\qquad\qquad = \int_\varepsilon^{\frac{1}{2}}  \sum_{l=-n}^{n}\frac{2i \sin (w \phi_{l,k})}{w^{1+\beta}} \widehat{f}(l) \widehat{g}(k-l) \mathrm{d}w \nonumber\\
	&\qquad\qquad =  \sum_{l=-n}^{n} \widehat{f}(l) \widehat{g}(k-l) 2i \int_\varepsilon^{\frac{1}{2}} \frac{ \sin (w \phi_{l,k})}{w^{1+\beta}}  \mathrm{d}w \nonumber\\
	&\qquad\qquad = \sum_{l=-n}^{n} \widehat{f}(l) \widehat{g}(k-l) 2i \phi_{l,k}^\beta \int_{\phi_{l,k}\varepsilon}^{\frac{\phi_{l,k}}{2}} \frac{ \sin (\widetilde{w})}{\widetilde{w}^{1+\beta}}  \mathrm{d}\widetilde{w} \nonumber\\
	&\qquad\qquad = \sum_{l=-n}^{n} \widehat{f}(l) \widehat{g}(k-l) 2i \phi_{l,k}^\beta (\Si_\beta\left(\tfrac{\phi_{l,k}}{2}\right) - \Si_\beta(\phi_{l,k}\varepsilon)),
	\end{align}
	where $\Si_\beta (x) := \int_0^x \frac{\sin (t)}{t^{1+\beta}}dt$. 
	Hence we can deduce  from the previous computations in \eqref{formula:pnqnepsilonI} and \eqref{formula:pnqnepsilonII} 
	\begin{align}\label{formula:Hfouriercoeffpnqn}
	\left|\widehat{H_{s_1,s_2,\beta}^\varepsilon(p_n,q_n)}(k) \right| 
	& \leq \sum_{l=-n}^{n} |\widehat{f}(l)| | \widehat{g}(k-l)| 2 |\phi_{l,k}|^\beta \left(\left|\Si_\beta\left(\frac{\phi_{l,k}}{2}\right) \right| + \left| \Si_\beta(\phi_{l,k}\varepsilon) \right| \right) \nonumber\\
	& \leq \tilde{M} \sum_{l=-n}^{n} |\widehat{f}(l)| | \widehat{g}(k-l)| |\phi_{l,k}|^\beta, 
	\end{align}
	where $\tilde{M} := 4 \sup_{\phi \in [0,\infty [} \Si_\beta (\phi) <\infty$ due to \cite[Lem.~2.4]{R12}. By \eqref{formula:Hfouriercoeffpnqn} and
	\begin{align*}
	|\phi_{k,l}|^\beta = |2\pi (ls_1 +(k-l)s_2)|^\beta \leq 2 \pi  |l|^{\beta} |k-l|^{\beta} \leq 2\pi (l^2 +1)^{\frac{\beta}{2}} ((k-l)^2+1)^{\frac{\beta}{2}},
	\end{align*}
	we conclude 
	\begin{align}\label{formula:Hfouriercoeffpnqnfinal}
		\left|\widehat{H_{s_1,s_2,\beta}^\varepsilon(p_n,q_n)}(k) \right| \leq & \tilde{M} \sum_{l=-n}^{n} |\widehat{f}(l)| | \widehat{g}(k-l)| |\phi_{l,k}|^\beta  \nonumber\\
		\leq & M \sum_{l=-n}^{n} (l^2 +1)^{\frac{\beta}{2}}  |\widehat{f}(l)| ((k-l)^2+1)^{\frac{\beta}{2}} | \widehat{g}(k-l)| \nonumber\\
		= & M \sum_{l=-n}^{n} (l^2 +1)^{\frac{\beta}{2}}  |\widehat{p_n}(l)| ((k-l)^2+1)^{\frac{\beta}{2}} | \widehat{q_n}(k-l)|,
	\end{align}
	where $M:= 2\pi \tilde{M} $. In addition, we obtain by \eqref{formula:Hfouriercoeffpnqnfinal} and the elementary estimates
	\begin{align*}
		(k^2+1)^{\frac{m}{2}} \leq 2^m \left((l^2+1)^{\frac{m}{2}} + ((k-l)^2+1)^{\frac{m}{2}} \right),
	\end{align*}
	which are an immediate consequence of $|k| \leq 2 \max\{|l|, |k-l|\}$, that
	\begin{align}\label{formula:Hs1s2decompose}	
&	\left|(k^2 + 1)^{\frac{m}{2}}\widehat{H_{s_1,s_2,\beta}^\varepsilon(p_n,q_n)}(k) \right|  \nonumber \\
	 &\qquad \leq  MC(m) \sum_{l\in\Z} (l^2 + 1)^{\frac{m+\beta}{2}} \left|\widehat{p_n}(l)\right|  ((k-l)^2+1)^{\frac{\beta}{2}} \left| \widehat{q_n}(k-l) \right| \nonumber \\
	&\qquad\qquad     + MC(m) \sum_{l\in\Z}  (l^2 +1)^{\frac{\beta}{2}}  \left|\widehat{p_n}(l)\right| ((k-l)^2 + 1)^{\frac{m+\beta}{2}} \left| \widehat{q_n}(k-l) \right|,
	\end{align}
	where the positive constant $C(m) < \infty$ only depends on the parameter $m$. 
	This leads us to the idea to set the following series component-wise for all $k\in\Z$ and any $\lambda >0$
	\begin{align*}
	p_n^\lambda(k) &:= (k^2 + 1)^{\frac{m+\lambda}{2}} \left|\widehat{p_n}(k)\right| \textnormal{ and } \\
	q_n^\lambda(k) &:= (k^2 + 1)^{\frac{m+\lambda}{2}} \left|\widehat{q_n}(k)\right|
	\end{align*}
	such that we can rewrite \eqref{formula:Hs1s2decompose} as
	\begin{align*}
	\left|(k^2 + 1)^{\frac{m}{2}}\widehat{H_{s_1,s_2,\beta}^\varepsilon(p_n,q_n)}(k) \right| \leq MC(m) \left( (p_n^\beta*q_n^0) (k) +  (p_n^0*q_n^\beta ) (k)\right).
	\end{align*}
	Finally, we obtain the desired estimate \eqref{formula:BilinearHT} for $p_n$ and $ q_n$ by applying Lemma~\ref{lem:Young}, Lemma~\ref{lem:Sobolev}, and Sobolev's embedding theorem to estimate 
	\begin{align}\label{formula:mainestimatepnqn}
	\|H_{s_1,s_2,\beta}^\varepsilon(p_n,q_n)\|_{H^m} & = \|( (k^2 + 1)^{\frac{m}{2}}\widehat{H_{s_1,s_2}^\varepsilon(p_n,q_n)}(k))_{k\in\Z} \|_{\ell^2} \nonumber\\
	& \leq MC(m) \left( \| ((p_n^\beta*q_n^0) (k))_{k\in\Z} \|_{\ell^2}  + \|((p_n^0*q_n^\beta ) (k))_{k\in\Z} \|_{\ell^2} \right) \nonumber\\
	& \leq MC(m) \left(\|p_n^\beta\|_{\ell^2} \|q_n^0\|_{\ell^1} +  \|p_n^0\|_{\ell^1} \|q_n^\beta\|_{\ell^2} \right) \nonumber\\
	& \leq MC(m)C_0 \left(\|p_n\|_{H^{m+\beta}} \|q_n\|_{H^m} + \|p_n\|_{H^m} \|q_n\|_{H^{m+\beta}}\right) \nonumber\\
	& = C_H \|p_n\|_{H^{m+\beta}} \|q_n\|_{H^{m+\beta}} 
	,\end{align}
	where $C_H:= 2MC(m)C_0$ is a constant only depending on $m$ and $\beta$. 
	
	Finally, we get the statement \eqref{formula:BilinearHT} for $f$ and $g$ by passing to the limit $n\rightarrow \infty$. In particular, by using \eqref{formula:mainestimatepnqn}, we find  for any integer $n\geq 0$ that
	\begin{align}
	\|H_{s_1,s_2,\beta }^\varepsilon(p_{n},q_{n})\|_{H^m} & \leq C_H \|p_{n}\|_{H^{m+\beta}} \|q_{n}\|_{H^{m+\beta}} \nonumber\\
	&  \leq C_H \|f\|_{H^{m+\beta}} \|g\|_{H^{m+\beta}}  \label{formula:Hfg}
	\end{align}
	By  uniform convergence and by 	Lemma~\ref{lem:Hs1s2cont} we find that 
	\begin{align*}
	& \|H_{s_1,s_2,\beta}^\varepsilon(f,g) - H_{s_1,s_2,\beta}^\varepsilon(p_n,q_n)\|_\infty \\
	&\quad \leq \|H_{s_1,s_2,\beta}^\varepsilon(f,g) - H_{s_1,s_2,\beta}^\varepsilon(p_n,g)\|_\infty + \|H_{s_1,s_2,\beta}^\varepsilon(p_n,g) - H_{s_1,s_2,\beta}^\varepsilon(p_n,q_n)\|_\infty \\
	&\quad = \|H_{s_1,s_2,\beta}^\varepsilon(f-p_n,g)\|_\infty + \|H_{s_1,s_2,\beta}^\varepsilon(p_n,g-q_n)\|_\infty \\
	&\quad \leq C(\beta,\varepsilon) 
	\|f-p_n\|_\infty
	\|g\|_\infty + C(\beta, \varepsilon) \|p_n\|_\infty 
	\|g-q_n\|_\infty \\
	&\quad \to 0
	\end{align*}
	as $n\to \infty$, since 
	$\|f-p_n\|_\infty \to 0$ and $ \|g-q_n\|_\infty \to 0$ as $n\to \infty$ by our assumption \eqref{formula:apprfg}.	
	Moreover, by the uniqueness of the Fourier coefficients of $ H_{s_1,s_2,\beta}^\varepsilon(p_n,q_n) \in H^m (\RZ,\R) \subseteq L^2(\RZ,\R)$ we have, for all $k\in\Z$, that
	\begin{align*}
	\left|\widehat{H_{s_1,s_2,\beta}^\varepsilon(f,g)}(k)- \widehat{H_{s_1,s_2,\beta}^\varepsilon(p_n,q_n)}(k)\right| 
	\to 0 
	\end{align*}
	as $n\to \infty$. 
	Thus for all positive integers $N \geq 1$ we find that 
		\begin{align*}
	&\Big| \Big( \sum_{|k|\leq N} (1+|k|^2)^m|\widehat{H_{s_1,s_2,\beta}^\varepsilon(p_n,q_n)}(k)|^2 \Big)^{\frac{1}{2}} 
	\\	&\hspace{3cm} 
	- \Big( \sum_{|k|\leq N} (1+|k|^2)^m|\widehat{H_{s_1,s_2,\beta}^\varepsilon(f,g)}(k)|^2 \Big)^{\frac{1}{2}} \Big| 
	\to 0
	\end{align*}
	as $n\to\infty$. 
	Then   by \eqref{formula:Hfg} we conclude that 
	\begin{align*}
	\Big( \sum_{|k|\leq N} (1+|k|^2)^m|\widehat{H_{s_1,s_2,\beta}^\varepsilon(f,g)}(k)|^2 \Big)^{\frac{1}{2}} \leq C_H \|f\|_{H^{m+\beta}} \|g\|_{H^{m+\beta}} 
	\end{align*}
	which implies the desired estimate \eqref{formula:BilinearHT} for $f$ and $g$.
\end{proof}

\section{The form of the lower order remainder terms $R_1^\alpha$ and $R_2^\alpha$}\label{sec:remaining parts}

We show that the orthogonal projection of the truncated remainder terms of the decomposition of $\delta \E^\alpha$ found in Section~\ref{section:decomposition} can be expressed as multiple integrals of analytic functions. 

Let $0 < \varepsilon \leq \frac{1}{2}$, $x\in [0,1]$ and $\gamma \in C^\infty(\RZ,\R^n)$ be a simple closed  arc-length parametrized curve. 
We recall that the truncated remainder terms are given by 
\begin{align*}
(R_1^{\alpha,\varepsilon} \gamma)(x) &= \int_{|w| \in [\varepsilon, \frac{1}{2}]} (\gamma(x+w)-\gamma(x)-w\dot{\gamma}(x)) \left(\frac{1}{|\gamma(x+w)-\gamma(x)|^{\alpha +2}} - \frac{1}{ |w|^{\alpha +2}}\right)\mathrm{d} w, \\
(R_2^{\alpha,\varepsilon} \gamma)(x) &= \int_{|w| \in [\varepsilon, \frac{1}{2}]} \ddot{\gamma}(x)\left(\frac{1}{|\gamma(x+w)-\gamma(x)|^{\alpha}} - \frac{1}{ |w|^{\alpha }}\right)\mathrm{d} w
.\end{align*}

We begin with transforming the more elementary part $P^\perp_{\dot{\gamma}} R_2^{\alpha,\varepsilon} \gamma$.

\begin{thm}
	The term $P^\perp_{\dot{\gamma}} R_2^{\alpha,\varepsilon} \gamma$ can be re-written as multiple integral of the form 
	\begin{align*}
	(P^\perp_{\dot{\gamma}} R_2^\varepsilon \gamma )(x) =  
	\int_{|w|\in [\varepsilon, \frac{1}{2}]} 
	 \iiiint_{[0,1]^4} (s_1-s_2)^2 \mathsf{G}^{\alpha}_2(\gamma)(x)  \mathrm{d} \phi_1 \mathrm{d} \phi_2 \mathrm{d} s_1 \mathrm{d} s_2 \frac{\mathrm{d} w}{|w|^{\alpha-2}}  ,
	\end{align*}
	where 
	\begin{align*}
	\mathsf{G}^{\alpha}_2(\gamma)(x) 
	&= \mathsf{G}^{\alpha}_2(\gamma)(x;{s_1,s_2,\phi_1,\phi_2,w}) \\	
	&= 
	G^{\alpha}_2\Big(  
	\int_0^1 \dot{\gamma} (x+tw) \mathrm{d} t,
	\dot{\gamma}(x),  	\ddot{\gamma}(x+s_2w + (s_1-s_2)\phi_1 w), \\
	&\hspace{3cm}
	\ddot{\gamma}(x+s_2w + (s_1-s_2)\phi_2 w), 
	\ddot{\gamma}(x)  \Big) 
	\end{align*}	
	for some  analytic function $G_2^\alpha : \R^n \setminus \{0,1\} \times \R^{4n} \rightarrow \R^n$. 
\end{thm}

\begin{proof}
We begin by   computing the first term in the integrand of $R_2^{\alpha,\varepsilon} \gamma$. To do so we use the fundamental theorem of calculus and the arc-length parametrization of the curve $\gamma$ (as in the proof of \cite[Prop.~3.2]{BR13}) to get 
\begin{align}\label{formula:R2gammaint1}
&
\frac{1}{|\gamma(x+w)-\gamma(x)|^\alpha} -  \frac{1}{|w|^\alpha}  \nonumber\\
&\qquad= \frac{|w|^\alpha}{|\gamma(x+w)-\gamma(x)|^\alpha} \frac{1- \frac{|\gamma(x+w)-\gamma(x)|^\alpha}{|w|^\alpha}}{1- \frac{|\gamma(x+w)-\gamma(x)|^2}{w^2}}  \frac{1- \frac{|\gamma(x+w)-\gamma(x)|^2}{w^2}}{|w|^\alpha} \nonumber\\
& \qquad= g_2^\alpha\left(\int_0^1 \dot{\gamma} (x+tw) \mathrm{d} t \right) \frac{2\left(1- \frac{|\gamma(x+w)-\gamma(x)|^2}{w^2}\right)}{|w|^\alpha} \nonumber\\
&\qquad = g_2^\alpha\left(\int_0^1 \dot{\gamma} (x+tw) \mathrm{d} t \right) \frac{\iint_{[0,1]^2} |\dot{\gamma}(x+s_1w)-\dot{\gamma}(x+s_2w)|^2 \mathrm{d} s_1 \mathrm{d} s_2}{|w|^\alpha},
\end{align}
where $g_2^\alpha(x) = \frac{1}{2|x|^\alpha}\frac{1-|x|^\alpha}{1-|x|^2}$ for all $x\in \R^n \setminus\{0,1\}$ is analytic away from the origin. 
Since we have by the fundamental theorem of calculus
\begin{align}\label{formula:R2gammaint2}
& \frac{|\dot{\gamma}(x+s_1w)-\dot{\gamma}(x+s_2w)|^2}{w^2}   
=    \int_0^1 \int_0^1 (s_1 - s_2)^2 \nonumber \\
& \qquad \times \left< \ddot{\gamma}(x+s_2w + (s_1-s_2)\phi_1 w), \ddot{\gamma}(x+s_2w + (s_1-s_2)\phi_2 w) \right>_{\R^n} \mathrm{d} \phi_1 \mathrm{d} \phi_2 ,
\end{align}
we can express $R_2^{\alpha,\varepsilon} \gamma$ as 
\begin{align*}
 (R_2^{\alpha,\varepsilon}  \gamma) (x)  
 &=  \int_{|w|\in [\varepsilon, \frac{1}{2}]} \iiiint_{[0,1]^4} (s_1-s_2)^2  \nonumber \\
&\qquad \times  \widetilde{G_2^\alpha}\Big( 
\int_0^1 \dot{\gamma} (x+tw) \mathrm{d} t, 
\ddot{\gamma}(x+s_2w + (s_1-s_2)\phi_1 w), 
\\
&\hspace{3cm}
\ddot{\gamma}(x+s_2w + (s_1-s_2)\phi_2 w), 
\ddot{\gamma}(x)  \Big) 
  \mathrm{d} \phi_1 \mathrm{d} \phi_2 \mathrm{d} s_1 \mathrm{d} s_2  \frac{\mathrm{d} w}{|w|^{\alpha-2}} 
,\end{align*}
where $\widetilde{G_2^\alpha}: \R^n \setminus \{0,1\} \times \R^{3n}  \rightarrow \R^n$ defined by $\widetilde{G_2^\alpha}(a, x, y, z):= g_2^\alpha(a) \left<x,y \right>_{\R^n} z$ is an analytic function away from the origin in the first variable as well. If we apply the orthogonal projection $P^\perp_{\dot{\gamma}}$ on $ R_2^{\alpha,\varepsilon} \gamma$, we finally get the following  representation
\begin{align*}
(P^\perp_{\dot{\gamma}} R_2^{\alpha,\varepsilon} \gamma )(x) =  \int_{|w|\in [\varepsilon, \frac{1}{2}]}  \iiiint_{[0,1]^4} (s_1-s_2)^2 \mathsf{G}^{\alpha}_2(\gamma)(x) \mathrm{d} \phi_1 \mathrm{d} \phi_2 \mathrm{d} s_1 \mathrm{d} s_2 \frac{\mathrm{d} w}{|w|^{\alpha-2}}
,\end{align*}
where  
\begin{align*}
\mathsf{G}^{\alpha}_2(\gamma)(x) 
&= \mathsf{G}^{\alpha}_2(\gamma)(x;{s_1,s_2,\phi_1,\phi_2,w}) \\	
&= 
G^{\alpha}_2\Big(  
\int_0^1 \dot{\gamma} (x+tw) \mathrm{d} t,
\dot{\gamma}(x),  	\ddot{\gamma}(x+s_2w + (s_1-s_2)\phi_1 w), \\
&\hspace{3cm}
\ddot{\gamma}(x+s_2w + (s_1-s_2)\phi_2 w), 
\ddot{\gamma}(x)  \Big) 
\end{align*}	
and 
$G_2^\alpha : \R^n \setminus \{0,1\} \times \R^{4n} \rightarrow \R^n$ given by $G_2^\alpha (a,v,x,y,z) := P^\perp_{v} \widetilde{G_2^\alpha} (a,x,y,z)$
is clearly analytic.
\end{proof}

We use the previous computations to rewrite the orthogonal projection of the first remaining part.

\begin{thm}\label{thm:formR1}
		The term $P^\perp_{\dot{\gamma}} R_1^{\alpha,\varepsilon} \gamma$ can be re-written as multiple integral of the form 
	\begin{align*}
	(P^\perp_{\dot{\gamma}} R_1^\varepsilon \gamma )(x) =  
	\int_{|w|\in [\varepsilon, \frac{1}{2}]} 
	\iiiint_{[0,1]^4} (r_1-r_2)^2 \mathsf{G}^{\alpha}_1(\gamma)(x)  \mathrm{d} \psi_1 \mathrm{d} \psi_2 \mathrm{d} r_1 \mathrm{d} r_2 \frac{\mathrm{d} w}{|w|^{\alpha-2}}  ,
	\end{align*}
	where 
	\begin{align*}
	\mathsf{G}^{\alpha}_1(\gamma)(x) 
	&= \mathsf{G}^{\alpha}_1(\gamma)(x;{r_1,r_2,\psi_1,\psi_2,w}) \\	
	&= 
	G^{\alpha}_1\Big(\int_0^1 \dot{\gamma} (x+sw) \mathrm{d} s,
	\dot{\gamma}(x), 
	\ddot{\gamma}(x+r_2w + (r_1-r_2)\psi_1 w), \\
	&\hspace{3cm}	
	\ddot{\gamma}(x+r_2w + (r_1-r_2)\psi_2 w), 
	\int_0^1 \ddot{\gamma}(x +tw)(1-t) \mathrm{d}t \Big) 
	\end{align*}	
	for some  analytic function $G_1^\alpha : \R^n \setminus \{0,1\} \times \R^{4n} \rightarrow \R^n$. 
\end{thm}

\begin{proof}
By the integral form of the remainder of a first order Taylor approximation, we compute the integrand of $R_1^{\alpha,\varepsilon} \gamma$ as
\begin{align}\label{formula:R1integrand}
& (\gamma(x+w)-\gamma(x)-w\dot{\gamma}(x)) \left(\frac{1}{|\gamma(x+w)-\gamma(x)|^{\alpha +2}} - \frac{1}{ |w|^{\alpha +2}}\right) \nonumber\\
&\qquad =  w^2 \int_0^1 \ddot{\gamma}(x+tw)(1-t) \mathrm{d} t \ \left(\frac{1}{|\gamma(x+w)-\gamma(x)|^{\alpha +2}} - \frac{1}{ |w|^{\alpha +2}}\right) \nonumber\\
&\qquad =  \int_0^1 \ddot{\gamma}(x+tw)(1-t) \mathrm{d} t \ \left(\frac{w^2}{|\gamma(x+w)-\gamma(x)|^{\alpha +2}} - \frac{1}{ |w|^{\alpha}}\right).
\end{align}
Thus the last term in \eqref{formula:R1integrand} can be expressed as
\begin{align}\label{formula:R1integrand2}
& \frac{w^2}{|\gamma(x+w)-\gamma(x)|^{\alpha +2}} - \frac{1}{ |w|^{\alpha}} \nonumber\\
& \qquad =  \frac{|w|^{\alpha +2}}{|\gamma(x+w)-\gamma(x)|^{\alpha +2}} \left(\frac{1-\frac{|\gamma(x+w)-\gamma(x)|^{\alpha +2}}{|w|^{\alpha +2}} }{|w|^\alpha}\right) \nonumber\\
&\qquad  = \frac{|w|^{\alpha +2}}{|\gamma(x+w)-\gamma(x)|^{\alpha +2}} \frac{1-\frac{|\gamma(x+w)-\gamma(x)|^{\alpha +2}}{|w|^{\alpha +2}} }{1-\frac{|\gamma(x+w)-\gamma(x)|^{2}}{|w|^{2}}} \left(\frac{1-\frac{|\gamma(x+w)-\gamma(x)|^{2}}{|w|^{2}} }{|w|^\alpha}\right) \nonumber\\
&\qquad  =g_1^\alpha\left(\int_0^1 \dot{\gamma} (x+sw) \mathrm{d} s \right) \left(\frac{1-\frac{|\gamma(x+w)-\gamma(x)|^{2}}{|w|^{2}} }{|w|^\alpha}\right)
,\end{align}
where $g_1^\alpha(x) = \frac{1}{2|x|^{\alpha+2}}\frac{1-|x|^{\alpha+2}}{1-|x|^2}$ for all $x\in \R^n\setminus \{0,1\}$ is analytic away from the origin. By transfering the computations in \eqref{formula:R2gammaint1} and \eqref{formula:R2gammaint2} to the last term in \eqref{formula:R1integrand2} we find that  
\begin{align*}
&(R_1^{\alpha,\varepsilon} \gamma) (x) \\
&=  \int_{|w|\in [\varepsilon, \frac{1}{2}]}  \iiiint_{[0,1]^4} (r_1-r_2)^2 \\
&\quad\times
 \widetilde{G_1^\alpha} \bigg( \int_0^1 \dot{\gamma}(x+sw)\mathrm{d}s, 
  \ddot{\gamma}(x+r_2w + (r_1-r_2)\psi_1 w), 
  \\
 &\hspace{2cm}
 \ddot{\gamma}(x+r_2w + (r_1-r_2)\psi_2 w), 
 \int_0^1 \ddot{\gamma}(x+tw)(1-t) \mathrm{d} t \bigg) \mathrm{d} \psi_1 \mathrm{d} \psi_2 \mathrm{d} r_1 \mathrm{d} r_2 \mathrm{d} w
,\end{align*}
where $\widetilde{G_1^\alpha} : \R^n\setminus \{0,1\} \times \R^{3n} \rightarrow \R^n$ defined by $\widetilde{G_1^\alpha} (a,x,y,z) := g_1^\alpha(a) \left<x,y \right>_{\R^n} z$ is analytic away from 0 and 1 in the first variable.

Finally, by applying the orthogonal projection to $R_1^{\alpha,\varepsilon} \gamma$, we obtain
\begin{align*}
(P^\perp_{\dot{\gamma}} R_1^{\alpha,\varepsilon} \gamma )(x) =  \int_{|w|\in [\varepsilon, \frac{1}{2}]}  \iiiint_{[0,1]^4}  (r_1-r_2)^2   \mathsf{G}^{\alpha}_1(\gamma)(x) \mathrm{d} \psi_1 \mathrm{d} \psi_2 \mathrm{d} r_1 \mathrm{d} r_2 \mathrm{d} w
,\end{align*}
where 
\begin{align*}
\mathsf{G}^{\alpha}_1(\gamma)(x) 
&= \mathsf{G}^{\alpha}_1(\gamma)(x;{r_1,r_2,\psi_1,\psi_2,w}) \\	
&= 
G^{\alpha}_1\Big(\int_0^1 \dot{\gamma} (x+sw) \mathrm{d} s,
\dot{\gamma}(x), 
\ddot{\gamma}(x+r_2w + (r_1-r_2)\psi_1 w), \\
&\hspace{3cm}	
\ddot{\gamma}(x+r_2w + (r_1-r_2)\psi_2 w), 
\int_0^1 \ddot{\gamma}(x +tw)(1-t) \mathrm{d}t \Big) 
\end{align*}
and
$G_1^\alpha : \R^n \setminus \{0,1\} \times \R^{4n} \rightarrow \R^n$ given by $G_1^\alpha (a,v,x,y,z) := P^\perp_{v} \widetilde{G_2^\alpha} (a,x,y,z)$
is obviously analytic.
\end{proof}

\section{Proof of the main theorem by Cauchy's method of majorants}\label{sec:mainpart}

We now turn to the proof of Theorem~\ref{main}. Our strategy is to first establish a  recursive estimate for $\|\partial^l \gamma \|_{H^{\alpha -1} }$ from which we can infer, by Cauchy's method of majorants, the analyticity of the curve $\gamma$.

Let $m := 1 > \frac{1}{2}$ and $\gamma = (\gamma_1,\ldots, \gamma_n) :\RZ\rightarrow \R^n$ be a simple  closed  arc-length parametrized curve that is in the class $ C^\infty (\RZ,\R^n)$. If $\gamma$ is a critical point of $\E^\alpha + \lambda \Ell$, i.e.~if we have  
\begin{align*}
\delta \E^\alpha(\gamma;h) + \lambda \int_{\RZ} \langle \dot{\gamma}, \dot{h} \rangle \mathrm{d}x = 0
\end{align*}
for all $h\in H^{\frac{\alpha +1}{2},2}(\RZ,\R^n)$, 
then Theorem~\ref{thm:1stvar} and integration by parts imply that the curve $\gamma$ solves the Euler-Lagrange equation
\begin{align*}
( H^\alpha\gamma + \lambda \ddot{\gamma},h)_{L^2([0,1], \R^n)} = 0
\end{align*}
for all $h \in  H^{\frac{\alpha +1}{2},2}(\RZ,\R^n)$.
Thus by the decompositions \eqref{formula:H} and \eqref{formula:Htilde} of the gradient of O'Hara's knot energies, we conclude that
$$
H^\alpha\gamma = \alpha P^\perp_{\dot{\gamma}} Q^\alpha\gamma + 2 \alpha P^\perp_{\dot{\gamma}} R_1^\alpha\gamma  - 2  P^\perp_{\dot{\gamma}} R_3^\alpha
\gamma  + \lambda \ddot{\gamma} \equiv 0
$$
on $\RZ$. Moreover, by Corollary~\ref{kor:Ql} with $\beta:= \alpha -2 \in (0,1)$ and the triangle inequality, it follows for any integer $l\geq 0$ that 
\begin{align}\label{formula:partiallgamma}
\|\partial^{l+3} \gamma \|_{H^{1+\beta} } & \leq \widetilde{C} \|\partial^l Q^\alpha\gamma\|_{H^1} \nonumber\\
& = \widetilde{C} \left\| \partial^l \left( \alpha P^T_{\dot{\gamma}} Q^\alpha\gamma  - 2\alpha P^\perp_{\dot{\gamma}} R_1^\alpha\gamma  + 2 P^\perp_{\dot{\gamma}} R_2^\alpha \gamma - \lambda \ddot{\gamma} \right) \right\|_{H^1}   \nonumber\\
& \leq 6\widetilde{C}    ( 
\|\partial^l P^T_{\dot{\gamma}} Q^\alpha\gamma\|_{H^1} 
+ \|\partial^l P^\perp_{\dot{\gamma}} R_1^\alpha\gamma \|_{H^1}  
\nonumber\\ &\hspace{2cm}
+ \|\partial^l P^\perp_{\dot{\gamma}} R_3^\alpha\gamma \|_{H^1} 
+ \lambda \|\partial^{l+2} \gamma \|_{H^1}  ).
\end{align}
In order to get a recursive estimate on $\|\partial^l \gamma \|_{H^{1+\beta} }$ we aim to derive suitable estimates for the first three terms on the right-hand side of \eqref{formula:partiallgamma}. 
To do  so for the tangential part of $Q^\alpha$, we use the fractional Leibniz rule  from Section \ref{sec:bilinearHilberttransform}
in the following:

\begin{lem}\label{est:Q}
	Let $l\in\N_0$, $\beta:= \alpha -2 \in (0,1)$ for some $2 < \alpha < 3$ and $0<\varepsilon\leq \frac 12$. Then there exists a positive constant $C_{Q^\alpha}$ independent of $\varepsilon$, $\gamma$ and $l$ such that 
	\begin{align*}
	&\|\partial^l P^T_{\dot{\gamma}} Q^{\alpha,\varepsilon}\gamma\|_{H^1} \\
	&\qquad \leq  C_{Q^\alpha} \sum_{k_1=0}^l \sum_{k_2=0}^{k_1} \binom{l}{k_1} \binom{k_1}{k_2}  \|\partial^{l-k_1+2} \gamma \|_{H^{1+\beta}} \|\partial^{k_1-k_2+2} \gamma \|_{H^{1+\beta}} \|\partial^{k_2 +1}\gamma \|_{H^1}. 
	\end{align*}
\end{lem}
\begin{proof}
	Using formula \eqref{formula:compTQ} and the Leibniz rule twice yields
	\begin{align*}
	&\partial^l  \Big(P^T_{\dot{\gamma}} Q^{\alpha,\varepsilon}\gamma (x)\Big)\\
	&=  \partial^l \Big(2  \int_{|w|\in [\varepsilon, \frac{1}{2}]}  \iint_{[0,1]^2} (1-t) (-t) \frac{\left<\ddot{\gamma}(x+tw), \ddot{\gamma}(x+stw)\right>_{\R^n}}{w |w|^{\alpha-2}} \dot{\gamma}(x) \mathrm{d}s \mathrm{d}t \mathrm{d}w \Big) \\
	&=   2  \int_{|w|\in [\varepsilon, \frac{1}{2}]}  \iint_{[0,1]^2} (1-t) (-t) \partial^l \left(\frac{\left<\ddot{\gamma}(x+tw), \ddot{\gamma}(x+stw)\right>_{\R^n}}{w|w|^{\beta}} \dot{\gamma}(x)\right) \mathrm{d}s \mathrm{d}t \mathrm{d}w \\
	&=  2  \int_{|w|\in [\varepsilon, \frac{1}{2}]} \iint_{[0,1]^2} (1-t) (-t) \\
	& \quad \times \sum_{k_1=0}^l \sum_{k_2=0}^{k_1} \binom{l}{k_1} \binom{k_1}{k_2} \frac{\left<\partial^{l-k_1+2} \gamma(x+tw), \partial^{k_1-k_2+2} \gamma (x+stw)\right>_{\R^n}}{w|w|^{\beta}} \partial^{k_2 +1}\gamma(x) \mathrm{d}s \mathrm{d}t \mathrm{d}w \\ 
	&=  2 \iint_{[0,1]^2} (1-t) (-t)  \\
	& \quad \times \sum_{k_1=0}^l \sum_{k_2=0}^{k_1} \binom{l}{k_1} \binom{k_1}{k_2}  \sum_{k=1}^n H_{t,st, \beta}^\varepsilon(\partial^{l-k_1+2} \gamma_k, \partial^{k_1-k_2+2} \gamma_k )(x) \partial^{k_2 +1}\gamma(x) \mathrm{d}s \mathrm{d}t,
	\end{align*}
	where the smoothness of the integrand allows interchanging of integrals and derivative. 
Then for all $1\leq m \leq n$
	we obtain the component-wise estimate 
	\begin{align}\label{formula:compQ}
	& \|\partial^l (P^T_{\dot{\gamma}} Q^{\alpha,\varepsilon}\gamma)_m\|_{H^1} \nonumber\\
	& \leq  2 \iint_{[0,1]^2} |1-t| |t|    \nonumber\\
	& \quad   \times \sum_{k_1=0}^l \sum_{k_2=0}^{k_1} \binom{l}{k_1} \binom{k_1}{k_2}  \sum_{k=1}^n C_{1} \| H_{t,st,\beta}^\varepsilon(\partial^{l-k_1+2} \gamma_k, \partial^{k_1-k_2+2} \gamma_k )\|_{H^1} \|\partial^{k_2 +1}\gamma_m \|_{H^1} \mathrm{d}s \mathrm{d}t \nonumber\\
	& \leq  2 C_{1} \iint_{[0,1]^2} |1-t| |t|   \nonumber\\
	& \quad  \times \sum_{k_1=0}^l \sum_{k_2=0}^{k_1} \binom{l}{k_1} \binom{k_1}{k_2}  \left( \sum_{k=1}^n C_H\|\partial^{l-k_1+2} \gamma_k \|_{H^{1+\beta}} \|\partial^{k_1-k_2+2} \gamma_k \|_{H^{1+\beta}} \right) \|\partial^{k_2 +1}\gamma_m \|_{H^1}  \mathrm{d}s \mathrm{d}t  \nonumber\\
	& \leq  2 C_{1} C_H \sum_{k_1=0}^l \sum_{k_2=0}^{k_1} \binom{l}{k_1} \binom{k_1}{k_2}  \|\partial^{l-k_1+2} \gamma \|_{H^{1+\beta}} \|\partial^{k_1-k_2+2} \gamma \|_{H^{1+\beta}} \|\partial^{k_2 +1}\gamma_m \|_{H^1}  
	\end{align}
	via the Banach algebra property \eqref{prop:Banachalgebra} of $H^1$ and 
	the fractional Leibniz rule for each of the components (i.e.~Theorem~\ref{thm:BilinearHT}).
	The desired statement follows from \eqref{formula:compQ} and the equivalence between the $1$-norm and the $2$-norm with a finite positive constant $C_{Q^\alpha}:= 2 C_{1} C_H \sqrt{n} $.	
\end{proof}

Next we estimate the orthogonal projection of the truncated remaining parts $R_1^{\alpha,\varepsilon} \gamma$ and $R_2^{\alpha,\varepsilon} \gamma$. The important ingredients for this are the forms of $P^\perp_{\dot{\gamma}} R_1^{\alpha,\varepsilon} \gamma$ and $P^\perp_{\dot{\gamma}} R_2^{\alpha,\varepsilon} \gamma$ derived in 
Chapter~\ref{sec:remaining parts}, Faà di Bruno's formula and the Banach algebra property of $H^1$.

\begin{lem}\label{est:R1}
Let  $l\in\N_0$, $2 < \alpha < 3$, and consider the mapping $f = (f_{1},\ldots, f_{5n})  : \RZ \rightarrow \R^n \setminus \{0,1\} \times \R^{4n}$
given by 
$$
	f(x) := (
	\dot{\gamma} (x),
	\dot{\gamma} (x),
	\ddot{\gamma}(x),
	\ddot{\gamma}(x),
	\ddot{\gamma}(x))
	.$$
	Then there exist positive constants $C_{R_1^\alpha}$ and $r$   independent of $\varepsilon$, $\gamma$ and $l$   such that 
	the universal polynomials $p_l^{(5n)}$ from the multivariate form of Fa\`a di Bruno's formula.
\eqref{eqn:FdB}	
	satisfy
	\begin{align*}
	\|\partial^l P^\perp_{\dot{\gamma}} R_1^{\alpha,\varepsilon} \gamma \|_{H^1} 
	&   \leq C_{R_1^\alpha} 
	p_l^{(5n)}\Big(\Big\{\frac{(|\eta|+1)!}{r^{|\eta|+1}} \Big\}_{|\eta| \leq l},
	 \Big\{C_1\|f_{i}^{(j)}\|_{H^1} \Big\}_{
	\substack{
	1\leq i \leq 5n \\
           1\leq j \leq l\ \, }	 	 
	 }\Big)
	,\end{align*}
	where $C_1$ is the constant from the estimate \eqref{prop:Banachalgebra} with $m=1$.
\end{lem}
\begin{proof}
	By Theorem~\ref{thm:formR1}  we can express 
	$P^\perp_{\dot{\gamma}} R_1^{\alpha,\varepsilon} \gamma$ as
	\begin{align*}
	(P^\perp_{\dot{\gamma}} R_1^{\alpha,\varepsilon} \gamma )(x) = \int_{|w|\in [\varepsilon, \frac{1}{2}]}  \iiiint_{[0,1]^4}  (r_1-r_2)^2  G_1^\alpha (f_1(x)) \mathrm{d} \psi_1 \mathrm{d} \psi_2 \mathrm{d} r_1 \mathrm{d} r_2 \frac{\mathrm{d} w}{|w|^{\alpha-2}}
	\end{align*}
	where $G_1^\alpha : \R^n \setminus \{0,1\} \times \R^{4n} \rightarrow \R^n$ is an analytic function
	and
	\begin{align*}
	f_1(x) := \begin{pmatrix}
	\int_0^1 \dot{\gamma} (x+sw) \mathrm{d}s \\
	\dot{\gamma} (x) \\
	\ddot{\gamma}(x+r_2w+(r_1-r_2)\psi_1 w) \\
	\ddot{\gamma}(x+r_2w+(r_1-r_2)\psi_2 w))  \\
	\int_0^1 \ddot{\gamma}(x+tw) \mathrm{d}t
	\end{pmatrix}.
	\end{align*}
	Due to the smoothness of the integrand and the generalized Faà di Bruno's formula 
	\eqref {eqn:FdB}, 
	we component-wise find for all $1 \leq k\leq 5n$ that 
	\begin{align*}
	& \partial^l (P^\perp_{\dot{\gamma}} R_1^{\alpha,\varepsilon} \gamma )_k(x) \\
	& = \partial^l \int_{|w|\in [\varepsilon, \frac{1}{2}]}  \iiiint_{[0,1]^4}  (r_1-r_2)^2  G_{1,k}^\alpha (f_1(x)) \mathrm{d} \psi_1 \mathrm{d} \psi_2 \mathrm{d} r_1 \mathrm{d} r_2 \frac{\mathrm{d} w}{|w|^{\alpha-2}} \\
	& = \int_{|w|\in [\varepsilon, \frac{1}{2}]}  \iiiint_{[0,1]^4}  (r_1-r_2)^2  \partial^l G_{1,k}^\alpha (f_1(x)) \mathrm{d} \psi_1 \mathrm{d} \psi_2 \mathrm{d} r_1 \mathrm{d} r_2 \frac{\mathrm{d} w}{|w|^{\alpha-2}} \\
	& = \int_{|w|\in [\varepsilon, \frac{1}{2}]}  \iiiint_{[0,1]^4}  (r_1-r_2)^2   \\
	& \hspace{3em} \times p_l^{(5n)}\Big(\{(\partial^\eta G_1^\alpha)_k(f_1(x)) \}_{|\eta| \leq l}, \{f_{1,i}^{(j)}(x) \}_{
\substack{
1\leq i \leq 5n \\
           1\leq j \leq l\ \, }	 
	}\Big) \mathrm{d} \psi_1 \mathrm{d} \psi_2 \mathrm{d} r_1 \mathrm{d} r_2 \frac{\mathrm{d} w}{|w|^{\alpha-2}}.
	\end{align*}
	Moreover, by applying the $H^1$-norm on $\partial^l (P^\perp_{\dot{\gamma}} R_1^{\alpha,\varepsilon} \gamma )$ component-wise for any $1 \leq k \leq 5n$, we find via the    Banach algebra property of $H^1$ that 
	\begin{align*}
	&\|\partial^l (P^\perp_{\dot{\gamma}} R_1^{\alpha,\varepsilon} \gamma )_k\|_{H^1}  \\
	& \leq \int_{|w|\in [\varepsilon, \frac{1}{2}]}  \iiiint_{[0,1]^4}  |r_1-r_2|^2  \\
	& \hspace{2.5em}\times  
	\Big\|p_l^{(5n)}\Big(
	\{(\partial^\eta G_1^\alpha)_k(f_1) \}_{|\eta| \leq l}, \{f_{1,i}^{(j)} \}_{	
\substack{
1\leq i \leq 5n \\
           1\leq j \leq l\ \, }	           	
	}\Big) 
	\Big\|_{H^1} \mathrm{d} \psi_1 \mathrm{d} \psi_2 \mathrm{d} r_1 \mathrm{d} r_2 \frac{\mathrm{d} w}{|w|^{\alpha-2}} \\
	& \leq  \int_{|w|\in [\varepsilon, \frac{1}{2}]}  \iiiint_{[0,1]^4}  |r_1-r_2|^2   \\
	& \hspace{2.5em}
	\times  p_l^{(5n)}\Big(
	\{ \left\|(\partial^\eta G_1^\alpha)_k(f_1)\right\|_{H^1} \}_{|\eta| \leq l}, \{C_1\|f_{1,i}^{(j)}\|_{H^1} \}_{
\substack{
1\leq i \leq 5n \\
           1\leq j \leq l\ \, }	 		
	}\Big) 
	\mathrm{d} \psi_1 \mathrm{d} \psi_2 \mathrm{d} r_1 \mathrm{d} r_2 \frac{\mathrm{d} w}{|w|^{\alpha-2}} 
	,\end{align*}
	where the constant $C_1$ appears only in front of $\|f_{1,i}^{(j)}\|_{H^1}$ because  of the one-homogeneity of $p_l^{(5n)}$ in its first components. 
	
	For any $1\leq j \leq l$ we also observe    that	
$	\|f_{1,k}^{(j)}\|_{H^1}  \leq \|f_k^{(j)}\|_{H^1} $
whenever  $1\leq k\leq n$, $ 4n+1\leq k \leq 5n$ 
and that 
$\|f_{1,k}^{(j)}\|_{H^1}  = \|f_k^{(j)}\|_{H^1}$
whenever $n+1\leq k\leq 4n$.
Moreover, $\|f_k^{(j)}\|_{H^1}$ does not depend on any of the parameters $w,\psi_i, r_i, s, t$ for all $1\leq k\leq 5n$ and $1\leq j \leq l$. 
In addition, since $G_1^\alpha$ is an analytic function, we can deduce from the proof of \cite[Theor.~7.2]{BV19} that 
$$
	\left\|(\partial^\eta G_1^\alpha)(f_1)\right\|_{H^1}  
	\leq \widetilde{C}_{R_1^\alpha}  \frac{(|\eta|+1)!}{r^{|\eta|+1}}
$$
for some positive constant $\widetilde{C}_{R_1^\alpha} $ independent of $\varepsilon$, $\gamma$ and $l$. 

   Now using the elementary estimate $\int_{|w|\in [\varepsilon, \frac{1}{2}]} \frac{\mathrm{d} w}{|w|^{\alpha-2}} \leq \frac{2}{3-\alpha} < \infty$ since $2 < \alpha < 3$ and the fact that $p_l^{(5n)}$ is one-homogeneous in the first components, we finally obtain
$$
	 \|\partial^l (P^\perp_{\dot{\gamma}} R_1^\varepsilon \gamma )\|_{H^1} 
	 \leq C_{R^\alpha_1} \, p_l^{(5n)}
	 \Big(\Big\{  \frac{(|\eta|+1)!}{r^{|\eta|+1}}\Big\}_{|\eta| \leq l}, \Big\{\|C_1f_{i}^{(j)}\|_{H^1} \Big\}_{
	 \substack{
1\leq i \leq 5n \\
           1\leq j \leq l\ \, }	
	 }\Big) ,
$$
	where $C_{R^\alpha_1} := \widetilde{C}_{R_1^\alpha}  \frac{2}{3-\alpha}$ is finite and the right-hand side is independent of $\varepsilon$.
\end{proof}

The orthogonal projection of the second remainder term can be estimated along the lines of the previous Lemma~\ref{est:R1} to get the following:

\begin{lem}\label{est:R2}
	Let $l\in\N_0$ and $f = (\dot{\gamma},\dot{\gamma},\ddot{\gamma}, \ddot{\gamma}, \ddot{\gamma})$ as in Lemma~\ref{est:R1}. Then there exist positive constants $C_{R_2^\alpha}$ and $s$ that are independent of $\varepsilon$, $\gamma$ and $l$ such that 	
	\begin{align*}
	&\|\partial^l P^\perp_{\dot{\gamma}} R_2^{\alpha,\varepsilon} \gamma \|_{H^1} \leq  C_{R_2^\alpha} \, p_l^{(5n)}\Big(\Big\{  \frac{(|\eta|+1)!}{s^{|\eta|+1}} \Big\}_{|\eta| \leq l}, \Big\{C_1\|f_{i}^{(j)}\|_{H^1} \Big\}_{
		 \substack{
1\leq i \leq 5n \\
           1\leq j \leq l\ \, }	
	}\Big) 
	,\end{align*} 
	where $p_l^{(5n)}$ is the same universal polynomial as in Lemma~\ref{est:R1}.
\end{lem}

As a  direct consequence of the proof of \cite[Theor.~7.4]{BV19}
the estimates from Lemmas~\ref{est:Q}, \ref{est:R1} and \ref{est:R2} also hold for the limit case $\varepsilon \downarrow 0$.

\begin{thm}\label{thm:estimates}
	Let $l\in\N_0$ and $f = (\dot{\gamma},\dot{\gamma}, \ddot{\gamma}, \ddot{\gamma}, \ddot{\gamma})$ as in Lemma~\ref{est:R1}. Then there exist positive constants $r$, $s$, $C_{Q^\alpha}$, $C_{R_1^\alpha}$ and $C_{R_2^\alpha}$ which are independent of $\gamma$ and $l$ such that 
	\begin{align*}
	\|\partial^l P^T_{\dot{\gamma}} Q^\alpha\gamma\|_{H^1}  & \leq  C_{Q^\alpha} \sum_{k_1=0}^l \sum_{k_2=0}^{k_1} \binom{l}{k_1} \binom{k_1}{k_2}  \|\partial^{l-k_1+2} \gamma \|_{H^{1+\beta}} \|\partial^{k_1-k_2+2} \gamma \|_{H^{1+\beta}} \|\partial^{k_2 +1}\gamma \|_{H^1} \\
	\|\partial^l P^\perp_{\dot{\gamma}} R_1^\alpha \gamma \|_{H^1}  
	& \leq C_{R_1^\alpha} \, 
	p_l^{(5n)}\Big(\Big\{\frac{(|\eta|+1)!}{r^{|\eta|+1}} \Big\}_{|\eta| \leq l}, 
	\Big\{C_1\|f_{i}^{(j)}\|_{H^1} \Big\}_{
				 \substack{
1\leq i \leq 5n \\
           1\leq j \leq l\ \, }	
           } \Big)\\
\|\partial^l P^\perp_{\dot{\gamma}} R_2^\alpha \gamma \|_{H^1} 
	&  \leq  C_{R_2^\alpha} \,
	p_l^{(5n)}\Big(\Big\{  \frac{(|\eta|+1)!}{s^{|\eta|+1}} \Big\}_{|\eta| \leq l}, 
	\Big\{C_1\|f_{i}^{(j)}\|_{H^1} \Big\}_{
			 \substack{
1\leq i \leq 5n \\
           1\leq j \leq l\ \, }	
           }\Big)  
,	\end{align*} 
where $\beta := \alpha -2$ for some $2 < \alpha < 3$ and  $p_l^{(5n)}$ is  the universal polynomial  from Lemma~\ref{est:R1}.
\end{thm}

We are now in a position to give the proof of the main theorem. 

\begin{proof}[Proof of Theorem~\ref{main}]
Let $\beta = \alpha -2$ for $2 <\alpha < 3$.
Now suppose
$\gamma = (\gamma_1,\ldots, \gamma_n):\RZ\rightarrow \R^n$
	is  a closed simple arc-length parametrized curve in the class $ C^\infty (\RZ,\R^n)$ which is a stationary point of  $\E^\alpha + \lambda \Ell$.  	
Using the smoothness of the curve $\gamma$, we introduce an auxiliary function $f:\RZ\rightarrow \R^{5n}$
	given by 
$$
	f(x):= (
	\dot{\gamma}(x) ,
	\dot{\gamma}(x) ,
	\ddot{\gamma}(x) ,
	\ddot{\gamma}(x) , 
	\ddot{\gamma}(x))
$$	
	and define 
	$$a_l :=C_1  \|\partial^l f\|_{H^{1+\beta}}$$
	 for all integers $l\geq 0$. If we can establish the existence of finite positive constants $\delta$ and $C_\gamma$ such that 
	\begin{align*}
	a_l \leq C_\gamma \frac{l!}{\delta^l} 
	\end{align*}
	we immediately obtain the analyticity of the curve $\gamma$ on $\RZ$ by Corollary~\ref{cor:TayloranalyticSob}.

In order to obtain the desired bounds on $a_l$ 
we first note that there exists a positive constant $\tilde{C}$ such that
	\begin{align*}
	\|\partial^{l+3} \gamma \|_{H^{1+\beta} } 
	&\leq \widetilde{C}   \Big( 
	\|\partial^l P^T_{\dot{\gamma}} Q^\alpha\gamma\|_{H^1} 
	+ \|\partial^l P^\perp_{\dot{\gamma}} R_1^\alpha\gamma \|_{H^1} \nonumber \\ &\qquad \qquad\quad 
	+ \|\partial^l P^\perp_{\dot{\gamma}} R_3^\alpha\gamma \|_{H^1} 
	+ \lambda \|\partial^{l+2} \gamma \|_{H^1} \Big)
	\end{align*}
by the criticality of the curve $\gamma$, the first variation formula of $\E^\alpha$ of  Theorem~\ref{thm:1stvar}, 
our decomposition of the first variation of $\E^\alpha + \lambda \Ell$ and Corollary~\ref{kor:Ql}.  
	Next we deduce from Theorem~\ref{thm:estimates} and the Sobolev embedding of $H^{1+\beta} \subseteq H^1$ that
	\begin{align*}
	a_{l+1} 
	& =  C_1 \|\partial^{l+1} f\|_{H^{1+\beta}} \\
	&  \leq C_1 (3 \|\partial^{l+3} \gamma \|_{H^{1+\beta}} + 2 \|\partial^{l+2} \gamma \|_{H^{1+\beta}}) \\
	& \leq 3\widetilde{C} C_1 \Big[
	\|\partial^l P^T_{\dot{\gamma}} Q^\alpha\gamma\|_{H^1} 
	+ \|\partial^l P^\perp_{\dot{\gamma}} R_1^\alpha\gamma \|_{H^1} \nonumber\\ &\hspace{2cm}
	+ \|\partial^l P^\perp_{\dot{\gamma}} R_2^\alpha\gamma \|_{H^1}  
	+ \lambda \|\partial^{l+2} \gamma \|_{H^1} \Big] + 2 C_1 a_l\\
	& \leq    3\widetilde{C} C_1 \Big[
	C_{Q^\alpha} \sum_{k_1=0}^l \sum_{k_2=0}^{k_1} \binom{l}{k_1} \binom{k_1}{k_2}  
	\|\partial^{l-k_1+2} \gamma \|_{H^{1+\beta}} 
	\|\partial^{k_1-k_2+2} \gamma \|_{H^{1+\beta}} \|\partial^{k_2 +1}\gamma \|_{H^1} \\
	&\hspace{2cm} 
	+ C_{R_1^\alpha} p_l^{(5n)}
	\Big(\Big\{ \frac{(|\eta|+1)!}{r^{|\eta|+1}} \Big\}_{|\eta| \leq l},
	 \Big\{C_1\|f_{i}^{(j)}\|_{H^1} \Big\}_{
	 			 \substack{
1\leq i \leq 5n \\
           1\leq j \leq l\ \, }	
           }\Big) \\
	& \hspace{2cm}\qquad 
	+  C_{R_2^\alpha} p_l^{(5n)}\Big(\Big\{\frac{(|\eta|+1)!}{s^{|\eta|+1}}  \Big\}_{|\eta| \leq l}, \Big\{C_1 \|f_{i}^{(j)}\|_{H^1} \Big\}_{
				 \substack{
1\leq i \leq 5n \\
           1\leq j \leq l\ \, }	
           }\Big) \\
           &\hspace{2cm}\qquad\qquad  
           + \lambda \|\partial^{l+2} \gamma \|_{H^1} \Big]
            + 2 C_1 a_l.
	\end{align*}
	Therefore,   by the embedding of  $H^{1+\beta} \subseteq H^1$ again, there exist positive constants $C > 1$ and $r_\gamma$ independent of $\gamma$ and $l$ such that
	\begin{align}\label{formula:al+1}
	a_{l+1} 
	& \leq   C \Big[
	\sum_{k_1=0}^l \sum_{k_2=0}^{k_1} \binom{l}{k_1} \binom{k_1}{k_2}  
	a_{l-k_1} a_{k_1-k_2} a_{k_2} \nonumber\\ &\hspace{3cm}
	 + p_l^{(5n)}\Big(\Big\{ \frac{(|\eta|+1)!}{r_\gamma^{|\eta|+1}}\Big\}_{|\eta| \leq l}, \Big\{ a_j \Big\}_{1\leq j \leq l}\Big)  
	 + a_l \Big] .
	\end{align}
	To apply Cauchy's method of majorants, we define a majorant $F: \R^{5n} \rightarrow \R^{5n}$ component-wise by setting
	\begin{align*}
	F_i(y) := C \Big( y_i^3 + \frac{1 }{(1 +\frac{ 5na_0 -  (y_1 + \cdots + y_{5n})}{r_\gamma})^2} + y_i\Big) 
	\end{align*}
	for any $y = (y_1,\ldots,y_{5n})\in \R^{5n}$ and all $1\leq i \leq 5n$. It is clearly analytic around $a_0 (1,\ldots, 1)$. Moreover, we  consider the following initial value problem 
	\begin{align}\label{formula:ODE}
	\begin{split}
	\dot{c}(t) &= F( c(t)), \\
	c(0) &= a_0 (1,\ldots, 1).
	\end{split}
	\end{align}
	 and note that the ODE has a unique  smooth solution $c:	 (-\varepsilon, \varepsilon) \to \R^{5n}$
	that is also analytic around $0$ for some $\varepsilon >0$  by Theorem~\ref{CK}.
Hence  we can write $c=c(t)$ as a Taylor series in a neighbourhood of $0$,  i.e.
	\begin{align*}
	c (t):= \sum_{k=0}^{\infty} \frac{\widetilde{a}_{k} (1,\ldots, 1)}{k!}  t^k 
	\end{align*}
	where $\widetilde{a}_{k} (1,\ldots, 1)= c^{(k)}(0)$. 
	Then by Theorem~\ref{thm:Tayloranalytic}
we can, for all $t\in B_{r_c}(0)$, bound 	
		\begin{align}\label{formula:canalytic}
	|c^{(l)}(t)| \leq M \frac{l!}{r_c^l}
	\end{align}
	for some 
	finite  positive constants $r_c$ and $M$ which are independent of any integer $l\geq 0$.
	By applying the Leibniz rule and Faà di Bruno's formula 
\eqref{eqn:FdB}  
we have that 
\begin{align*}	
	&\partial^l F_i(c(t)) \\
	&=  C \Big[
	\sum_{k_1=0}^l \sum_{k_2=0}^{k_1} \binom{l}{k_1} \binom{k_1}{k_2}   c_i^{(l-k_1)}(t)  c_i^{(k_1-k_2)}(t) c_i^{(k_2)}(t) \\
	&\quad
 +p_l^{(5n)}\Big(
 \Big\{ \frac{(|\eta|+1)!}{r_\gamma^{|\eta|+1}} 
 \frac{1}{\left(1 + \left( \frac{5na_0}{r_\gamma} -\frac{\left(c_1(t) + \cdots + c_{5n}(t)\right)}{r_\gamma}\right)\right)^{|\eta|+2}} 
 \Big\}_{|\eta| \leq l}, \Big\{ c_i^{(j)}(t) \Big\}_{
 				 \substack{
1\leq i \leq 5n \\
           1\leq j \leq l\ \, }	
           }\Big)\\
 	&\quad\qquad + c_i^{(l)} (t) \Big]
	\end{align*}	
 for any $1\leq i\leq 5n$ and $l\in\N_0$.
Then by considering the initial condition \eqref{formula:ODE} we find that 
	\begin{align}\label{formula:tildeal+1}
	\tilde a_{l+1} 
& = \partial^l F_i (c(0)) \nonumber\\
&=  C \Big[  
\sum_{k_1=0}^l \sum_{k_2=0}^{k_1} \binom{l}{k_1} \binom{k_1}{k_2}   
\tilde a_{l-k_1} \tilde a_{k_1-k_2} \tilde a_{k_2} \nonumber\\ &\hspace{2cm} 
	 + p_l^{(5n)}\Big(\Big\{ \frac{(|\eta|+1)!}{r_\gamma^{|\eta|+1}} \Big\}_{|\eta| \leq l}, \Big\{\tilde a_j \Big\}_{1\leq j \leq l}\Big) + \tilde a_l  \Big]
	.\end{align}
	Finally an induction argument obtained from comparing \eqref{formula:al+1} and \eqref{formula:tildeal+1}, together with the fact that Faà di Bruno's polynomials $p_l$ \eqref{eqn:FdB} have non-negative coefficients	
and the initial condition $a_0 = \widetilde{a}_0 $, gives
	\begin{align*}
	a_l  \leq \widetilde{a}_{l}. 
	\end{align*}
	Then by \eqref{formula:canalytic} we conclude that  $\|\partial^{l+1} \gamma \|_{H^{1+\beta}} \leq a_l \leq M\frac {l!}{r_C^l}$ for all $l\in\N_0$ which,  by Corollary~\ref{cor:TayloranalyticSob}, implies the analyticity of the curve $\gamma$.
\end{proof}

\end{document}